\documentclass[a4paper]{amsart}

\usepackage{amssymb,pstricks,amscd}

\usepackage[totalwidth=17cm,totalheight=24cm]{geometry}
\usepackage{graphicx}
\input psfig

\newcommand{\fb}[1]{{{#1}}}
\newcommand{\jms}[1]{{{#1}}}

\begin{document}

\newtheorem{cor}{Corollary}[section]
\newtheorem{theorem}[cor]{Theorem}
\newtheorem{prop}[cor]{Proposition}
\newtheorem{lemma}[cor]{Lemma}
\newtheorem{sublemma}[cor]{Sublemma}
\newtheorem{defi}[cor]{Definition}

\theoremstyle{remark}
\newtheorem{remark}[cor]{Remark}
\newtheorem{example}[cor]{Example}

\newcommand{\cD}{{\mathcal D}}
\newcommand{\FF}{{\mathcal F}}
\newcommand{\cH}{{\mathcal H}}
\newcommand{\cL}{{\mathcal L}}
\newcommand{\cM}{{\mathcal M}}
\newcommand{\cT}{{\mathcal T}}
\newcommand{\cML}{{\mathcal M\mathcal L}}
\newcommand{\cGH}{{\mathcal G\mathcal H}}
\newcommand{\C}{{\mathbb C}}
\newcommand{\D}{{\mathbb D}}
\newcommand{\N}{{\mathbb N}}
\newcommand{\R}{{\mathbb R}}
\newcommand{\Z}{{\mathbb Z}}
\newcommand{\Kt}{\tilde{K}}
\newcommand{\Mt}{\tilde{M}}
\newcommand{\dr}{{\partial}}
\newcommand{\tr}{\mbox{tr}}
\newcommand{\isom}{\mbox{Isom}}
\newcommand{\isomz}{\mbox{Isom}_{0,+}}
\newcommand{\vect}{\mbox{Vect}}
\newcommand{\kappab}{\overline{\kappa}}
\newcommand{\pib}{\overline{\pi}}
\newcommand{\Sigmab}{\overline{\Sigma}}
\newcommand{\gd}{\dot{g}}
\newcommand{\diff}{\mbox{Diff}}
\newcommand{\dev}{\mbox{dev}}
\newcommand{\devb}{\overline{\mbox{dev}}}
\newcommand{\devt}{\tilde{\mbox{dev}}}
\newcommand{\vol}{\mbox{Vol}}
\newcommand{\hess}{\mbox{Hess}}
\newcommand{\db}{\overline{\partial}}
\newcommand{\gammab}{\overline{\gamma}}
\newcommand{\Sigmat}{\tilde{\Sigma}}
\newcommand{\mut}{\tilde{\mu}}
\newcommand{\phit}{\tilde{\phi}}

\newcommand{\cunc}{{\mathcal C}^\infty_c}
\newcommand{\cun}{{\mathcal C}^\infty}
\newcommand{\dd}{d_D}
\newcommand{\dmin}{d_{\mathrm{min}}}
\newcommand{\dmax}{d_{\mathrm{max}}}
\newcommand{\Dom}{\mathrm{Dom}}
\newcommand{\dn}{d_\nabla}
\newcommand{\ded}{\delta_D}
\newcommand{\delmin}{\delta_{\mathrm{min}}}
\newcommand{\delmax}{\delta_{\mathrm{max}}}
\newcommand{\hmin}{H_{\mathrm{min}}}
\newcommand{\maxi}{\mathrm{max}}
\newcommand{\oL}{\overline{L}}
\newcommand{\oP}{{\overline{P}}}
\newcommand{\Ran}{\mathrm{Ran}}
\newcommand{\tgamma}{\tilde{\gamma}}
\newcommand{\cotan}{\mbox{cotan}}
\newcommand{\lambdat}{\tilde\lambda}
\newcommand{\St}{\tilde S}

\newcommand{\II}{I\hspace{-0.1cm}I}
\newcommand{\III}{I\hspace{-0.1cm}I\hspace{-0.1cm}I}
\newcommand{\HSt}{\tilde{\operatorname{HS}}}
\newcommand{\note}[1]
{\vspace{10pt}\par\noindent\begin{center}
 \psshadowbox[linecolor=red, shadowcolor=blue]
{\begin{minipage}{.90\textwidth}
#1\end{minipage}}
\end{center}\vspace{10pt}\par\noindent}

\newcommand{\new}[1]{{\color{blue} #1} }

\newcommand{\old}[1]{{\color{red}[ #1 ]}}

\newcommand{\op}{\operatorname}

\newcommand{\AdS}{\operatorname{AdS}}
\newcommand{\uAdS}{\widetilde{\operatorname{AdS}}}
\newcommand{\dS}{\operatorname{dS}}
\newcommand{\HH}{\mathbb H}
\newcommand{\PP}{\mathbb P}
\newcommand{\RR}{\mathbb R}
\newcommand{\uRP}{\widetilde{\R\PP}^1}
\newcommand{\rp}{\R\PP}
\newcommand{\HS}{\operatorname{HS}}
\newcommand{\SO}{\operatorname{SO}}
\newcommand{\cF}{\operatorname{\mathcal{F}}}
\newcommand{\kD}{\mathfrak{D}}

\newcommand{\reg}{\mathrm{reg}}
\newcommand{\link}{\Sigma}
\newcommand{\slice}{S}
\newcommand{\hol}{h}
\newcommand{\struct}{\mathcal U(g, T, \theta)}

\newcommand{\ch}{\mathrm{cosh}}
\newcommand{\sh}{\mathrm{sinh}}

\title[Collisions of particles]{Collisions of particles in locally AdS spacetimes II\\
Moduli of globally hyperbolic spaces}

\author[]{Thierry Barbot}
\address{Laboratoire d'analyse non lin\'eaire et g\'eom\'etrie\\
Universit\'e d'Avignon et des pays de Vaucluse\\
33, rue Louis Pasteur\\
F-84 018 Avignon, France}
\email{thierry.barbot@univ-avignon.fr}
\author[]{Francesco Bonsante}
\address{Dipartimento di Matematica dell'Universit\`a degli Studi di Pavia,
via Ferrata 1, 27100 Pavia, Italy}
\email{francesco.bonsante@unipv.it}
\author[]{Jean-Marc Schlenker}
\address{Institut de Math\'ematiques de Toulouse, 
UMR CNRS 5219 \\
Universit\'e Paul Sabatier\\
31062 Toulouse Cedex 9, France}
\email{schlenker@math.univ-toulouse.fr}
\thanks{T. B. and F. B. were partially supported by CNRS, ANR GEODYCOS. J.-M. S. was
partially supported by the A.N.R. programs RepSurf,
ANR-06-BLAN-0311, GeomEinstein, 06-BLAN-0154, and ETTT, 2009-2013}

\date{\today}

\begin{abstract}
We investigate globally hyperbolic 3-dimensional AdS manifolds containing ``particles'',
i.e., cone singularities of angles less than $2\pi$ along a time-like graph $\Gamma$. 
To each such space we associate a graph and a finite family of pairs of hyperbolic
surfaces with cone singularities. We show that this data is sufficient to recover
the space locally (i.e., in the neighborhood of a fixed metric). This is a partial
extension of a result of Mess for non-singular globally hyperbolic AdS manifolds.
\end{abstract}

\maketitle

\tableofcontents

\section{Introduction}

\subsection{The 3-dimensional AdS space}

The anti-de Sitter space, $AdS_n$, is a complete Lorentzian manifold of 
constant curvature $-1$. It can be defined as a quadric in the space 
$\R^{n-1,2}$, that is, $\R^{n+1}$ endowed with a symmetric bilinear form of
signature $(n-1,2)$:
$$ AdS_n = \{ x\in \R^{n-1,2} ~|~ \langle x,x\rangle = -1 \}~. $$
In some ways, $AdS_3$ can be considered as a Lorentz analog of hyperbolic
3-space. We are interested here in manifolds endowed with a geometric 
structure which, outside some singular locus, is locally isometric to
$AdS_3$.

\subsection{Globally hyperbolic AdS spacetimes}
\label{sub:GH-AdS} 

A Lorentz $3$-manifold $M$ is AdS if it is locally modeled on $AdS_3$. 
Such a manifold is {\it globally hyperbolic maximal compact} (GHMC) if 
it contains a closed, space-like surface $S$, if any inextendible time-like
curve in $M$ intersects $S$ exactly once, and if $M$ is maximal under this
condition (any isometric embedding of $M$ in an AdS 3-manifold satisfying
the same conditions is actually an isometry). Those manifolds can in some
respects be considered as Lorentz analogs of quasifuchsian hyperbolic 
3-manifolds.

Let $S$ be a closed surface of genus at least $2$. A well-known theorem 
of Bers \cite{bers} asserts that the space of quasifuchsian hyperbolic metrics on 
$S\times \R$ (considered up to isotopy) is in one-to-one correspondence 
with $\cT_S\times \cT_S$, where $\cT_S$ is the Teichm\"uller space of $S$.

\jms{
In his 1990 IHES preprint, published only in 2007 \cite{mess, mess-notes},
G. Mess 
}
discovered a remarkable analog of this 
theorem for globally hyperbolic maximal anti-de Sitter spacetimes:
the space of GHM AdS metrics on $S\times \R$ is also parameterized by 
$\cT_S\times \cT_S$. Both the Bers and the Mess results can be described 
as ``stereographic pictures'': the full structure of a 3-dimensional
constant curvature spacetime is encoded in a pair of hyperbolic metrics on a surface.

To understand this result more precisely, recall that the identity component of the isometry group of
$AdS_3$, $O_0(2,2)$, is isomorphic, up to finite index, to $PSL(2,\R)\times PSL(2,\R)$. 
Given a 3-dimensional globally hyperbolic maximal compact AdS spacetime $M$, it is homeomorphic 
to $S\times \R$, where $S$ is a closed surface which can be chosen to be any
Cauchy surface in $M$. Then $M$ is isometric to $\Omega/\hol(\pi_1(S))$, where 
$\Omega$ is a convex subset of $AdS_3$ and $\hol:\pi_1(S)\rightarrow O_0(2,2)$ is 
a homomorphism, which can be written as $(\hol_l,\hol_r)$ in the identification 
of $O_0(2,2)$ with $PSL(2,\R)\times PSL(2,\R)$. 

\begin{theorem}[Mess \cite{mess}] \label{tm:mess}
The homomorphisms $\hol_l$ and $\hol_r$ 
are holonomy representations of hyperbolic metrics on $S$ and can be identified
with points in the Teichm\"uller space $\cT_S$ of $S$. The map sending $M$
to $(\hol_l,\hol_r)\in \cT_S\times \cT_S$ is a homeomorphism.
\end{theorem}

\jms{
The key point in the proof given by Mess is to show that
$\hol_l$ and $\hol_r$ have maximal Euler number.}
\fb{By a  celebrated result of Goldman \cite{goldman}, this maximality}
\jms{implies that}
\fb{they are Fuchsian representations.
}

\subsection{Cone singularities}
\label{sub:cone} 

Cone singularities along curves have been studied often in hyperbolic geometry, see
e.g. \cite{CHK,BLP}.
A basic model space is the metric space $\HH^3_\theta, \theta\in (0,2\pi)$, defined
as follows. Let $\Delta$ be a geodesic in $\HH^3$, and let $P,P'$ be two half-planes
bounded by $\Delta$, such that the oriented angle between $P$ and $P'$
is $\theta$. Then $\HH^3_\theta$ is obtained by ``cutting out'' the part of 
$\HH^3$ bounded by $P$ and $P'$ (with angle $\theta$ along $\Delta$) and 
isometrically gluing the boundary half-planes by the isometry which is the
identity on $\Delta$. 

A hyperbolic cone-spacetime with singular locus a link is a metric space where
each point has a neighborhood isometric to a neighborhood of $\HH^3_\theta$, for
some $\theta\in (0,2\pi)$. A key rigidity result for closed 3-dimensional 
cone-spacetimes with singular locus a link and cone angles in $(0,2\pi)$
was proved by Hodgson and Kerckhoff \cite{HK}, and had a profound influence
on hyperbolic geometry in recent years, see e.g. \cite{brock-bromberg-evans-souto}. More recently this
rigidity result was extended to closed hyperbolic cone-spacetimes with 
singular set a graph, see \cite{mazzeo-montcouquiol,weiss:09}

\subsection{AdS spacetimes and particles}

3-dimensional AdS spacetimes were first studied as a lower-dimensional toy model
of gravity: they are solutions of Einstein's equation, with negative cosmological constant
but without matter. A standard way to add physical relevance to this model is to
consider in those AdS spacetimes some point particles, modeled by cone singularities
along time-like lines (see e.g. \cite{thooft1,thooft2}). 
They can be described as we did in Section \ref{sub:cone} in the case of hyperbolic geometry:
the geodesic $\Delta$ has to be a time-like geodesic in $AdS_3$ and $P$, $P'$ are time-like totally geodesic
half-planes bounding $\Delta$, with cone angle $\theta$. If we remove the region bounded by $P$ and $P'$ and 
glue the boundary half-plane by the unique AdS-isometry sending $P$ and $P'$ which is the identity on $\Delta$,
we obtain a spacetime containing a singular line.

Here we will call ``massive particle'' such a cone singularity, of
angle less than $2\pi$, along a time-like line. The condition that the angle is less
than $2\pi$ is usually made by physicists, who consider it as corresponding to the
positivity of mass. Here, as in \cite{colI}, it is also mathematically relevant. 

Globally hyperbolic AdS spaces
with such particles were considered in \cite{cone}, when the cone angles are less than
$\pi$. It was shown that a satisfactory extension of Theorem \ref{tm:mess} exists in this setting, 
with elements of the Teichm\"uller space of $S$ replaced by hyperbolic metrics
with cone singularities, with cone angles equal to the angles at the 
``massive particles''. There are corresponding results in the hyperbolic case \cite{qfmp,conebend}
where the Bers double uniformization theorem extends to quasifuchsian manifolds with ``particles'':
cone singularities along infinite lines, with angle in $(0,\pi)$.
Here, by contrast, we allow cone angles to go all the
way to $2\pi$. This is a crucial difference since it allows massive particles to ``interact''
in interesting ways.

More general types of particles, including cone singularities along time-like
or light-like lines and ``black hole'' singularities, are considered in \cite{colI},
where the reader can find a local description at the interaction points as well as
some global examples and related constructions. Here we focus on massive particles
only, our main goal is to show how the moduli space of globally hyperbolic AdS
metrics with interacting massive particles, on a given 3-spacetime, can be locally
parameterized by finite sequences of pairs of hyperbolic surfaces with cone singularities. This
can be considered as a first step towards an extension of the Bers-type result of Mess \cite{mess} quoted
above. However the collisions between particles \jms{mean} 
that the situation is much richer and more complex 
than for angles less than $\pi$ as considered in \cite{cone}, where no collision occurs.

\jms{
\subsection{Left and right metrics of spacial slices}

A precise definition of the spacetimes with collision considered here is 
introduced in Definition \ref{df:admissible}. It is basically a pile of ``spacial slices'', each
the product of a closed surface by an interval, containing particles but no collision,
so that the collisions occur on the common boundary of two adjacent slices. We require that
each such boundary surface contains a unique collision. In Section \ref{sub:maximal} we
introduce a notion of $m$-spacetime, where $m$ stands for ``maximal'', and prove that every
admissible spacetime embeds in a unique $m$-spacetime satisfying a natural condition 
(Lemmas \ref{mm:lem} and \ref{lm:m-unique}).

In Section \ref{sc:holonomy} we study the holonomy representation of admissible spacetimes. 
We define a notion of admissible holonomy, and prove that small (admissible) deformations of 
an admissible spacetime are parameterized by small (admissible) deformations of its holonomy
representation (Theorem \ref{struct:thm}). 

This leads in Section \ref{sc:leftright} to the analysis of the left and right holonomies
of a spacetime with collisions. In the non-singular setting, those left and right holonomies
can be defined, as in \cite{mess}, using the decomposition of the identity component of 
$SO(2,2)$ as the product of two copies of $PSL(2,\R)$. When cone singularities are present,
however, this viewpoint is useful but sometimes not very convenient. 
In Section \ref{ssc:leftright} we introduce two flat connections on a 3-dimensional AdS manifold
and use them in Section \ref{ssc:lrmetrics} to construct a metric, locally isometric to 
$\HH^2\times \HH^2$, on the space of time-like geodesics in a 3-dimensional AdS manifold.

In Section \ref{ssc:transverse} we define a notion of ``transverse vector
field'' along a space-like surface in an AdS manifold: it is basically a time-like
vector field which behaves well at the particles and does not ``rotate'' too quickly. 
Using such a vector field along a space-like surface $S$, and the metric defined above on 
the space of time-like geodesics, it is possible to define
on $S$ two hyperbolic metrics, with cone singularities of equal angle at the intersection
with the particles (see Proposition \ref{pr:mulr}). 

Those two metrics are called the left and right
hyperbolic metrics of the slice, they do not depend on the choice of $S$ or of the
transverse vector field (see Lemma \ref{lm:surface}) and their holonomy representations are the left and right
components of the holonomy representation of the AdS structure.

This construction based on a transverse vector field is more
general and more flexible than that used in \cite{minsurf,cone}, which used space-like
surfaces with more stringent constraints. The added flexibility is necessary here to
understand how the two metrics change when the surface $S$ moves across a particle interaction.
}

\subsection{Stereographic picture of spacetimes with colliding particles}

\jms{To each AdS spacetime with interacting particles} is associated a sequence (or more precisely 
a graph) of ``spacial slices'', each corresponding to a domain where
no interaction occurs. To each slice we associate 
\jms{as explained above} a ``stereographic
picture'': a ``left'' and a ``right'' hyperbolic metric, both with
cone singularities of the same angles, which together are sufficient
to reconstruct the spacial slice. This construction is described in Section 
\ref{ssc:graph}.

A 
new but apparently natural notion occurs, that
of a ``good'' spacial slice: one containing space-like surfaces 
with a transverse vector field.
There are examples of spacetimes containing a ``good'' spacial slice 
which \jms{stop} being ``good'' after a particle interaction. A GHMC
AdS spacetimes with particles is ``good'' if it is made of good
space-like slices, see Definition \ref{df:good}.

Adjacent spacial slices are ``related'' by a particle interaction. 
In Section \ref{ssc:surgeries} we show that the left and right hyperbolic
metrics before and after the interaction in a good space-time are related by a 
surgery involving, for both the left and right metrics, 
the link of the interaction point: for both the left and right metrics,
a topological disk is found, isometric to a large enough disk in the past
component of the link of the interaction point, and each of those disks
is replaced (in a compatible way) by a large enough disk in the future
component of the link of the interaction point --- see 
Definition \ref{df:double} \jms{and Proposition \ref{pr:collision}}. Since the same 
surgery is done on
both the left and the right hyperbolic metrics, we use the term
``double surgery''.

As a consequence, to a good AdS space-time with particles, we can associate
two distinct pieces of information:
\begin{itemize}
\item a ``topological data'', namely the position of the singular graph
(the particles along with the interaction points) in the spacetime,
\item a ``geometric data'', where to each spatial slice is associated
a pair of hyperbolic surfaces with cone singularities (a ``stereographic picture''), 
and to each interaction point is associated a double surgery (see Definition \ref{df:geometric}).
\end{itemize}
This is developed in Section \ref{ssc:structure}.

\subsection{The stereographic picture is a complete description (locally)}

In Section \ref{sc:final} we show that this locally provides a complete description of
possible AdS spaces with interacting massive particles, i.e., 
given an AdS metric $g$ with interacting particles, a small neighborhood 
$g$ in the space of AdS metrics with interacting particles with the same singular graph is parameterized
by the admissible deformations of the topological geometric data associated
to the spacial slices. This is Theorem \ref{tm:main}, which can \fb{be informally}
formulated here as follows.

\begin{theorem} \label{tm:main-intro}
Let $M$ be an admissible AdS spacetime with interacting particles, and let $(T,G)$ be
the associated topological and geometric data. Then $M$ is uniquely determined
by $(T,G)$. Moreover, any small enough deformation of $G$ corresponds to 
an admissible AdS spacetime with interacting particles close to $M$, with the same
topological data $T$. 
\end{theorem}

This statement is obviously informal since we did not introduce yet a number of
notions necessary to make it precise, in particular concerning the space of 
topological and geometric data, etc. 
The precise form of the statement can be
found below as Theorem \ref{tm:main}.

In the non-singular case (Theorem \ref{tm:mess}) the parameterization of the
space of GHMC AdS spacetimes by 2-dimensional data --- a pair of hyperbolic metrics ---
is not only local, but global. It is of course natural to wonder whether it could be
possible to extend Theorem \ref{tm:main} to a global existence theorem of an
admissible AdS spacetime with particles having a given topological and geometric
data. This question leads to new issues that we do not consider here.

\subsection{Contents}

The exposition below goes from the more general arguments to those tailored more
specifically for AdS spacetimes with interacting particles. 

In Section \ref{sc:space} we define the notion of admissible AdS spacetime with
particles occuring in Theorem \ref{tm:main-intro}, and prove basic statements on 
the extension of isometries on AdS spacetimes with interacting particles. 

In Section \ref{sc:holonomy} we consider more specifically the holonomy representation of admissible
AdS spacetimes with particles. The main result is Theorem \ref{struct:thm}, which states
that small deformations of the AdS structure are in one-to-one correspondence to ``admissible''
deformations of the holonomy representation.

In Section \ref{sc:leftright}, we define a notion of good spacelike slice
and define the left and right hyperbolic metrics associated to such a slice. 
We then consider more specifically the properties of those left and right
metrics in relation with a collision point in the boundary of the
spacelike slice.

Section \ref{sc:surgeries} deals with the change in the left and right metrics
when a collision happens. The central notion of double surgery
is introduced there. It is then possible to define precisely the topological
and geometric data associated to an AdS spacetime with interacting particles.

Section \ref{sc:final} contains the main result of the paper, Theorem \ref{tm:main} (which is the
same as Theorem \ref{tm:main-intro} but stated more precisely).

The appendix contains a more technical development which is not necessary 
for the proof of the main result but which should clarify, for the more
interested readers, the definition of a double surgery; it shows why this
definition, which could \fb{at} a first sight appear more complicated than necessary,
is actually relevant.

\jms{
\subsection*{Acknowledgements}

We are grateful to two anonymous referees for many helpful comments, which lead to 
a much better exposition.
}

\section{The space of maximal spacetimes with collisions}
\label{sc:space}

\subsection{Singular AdS spacetimes}

This paper is to some extend a continuation of \cite{colI}, where we studied the geometry of
3-dimensional AdS spacetimes with interacting particles. The particles considered in
\cite{colI} are more general than those under consideration here, since they are cone
singularities on a space-like, a light-like, or a time-like curve, as well as more
exotic objects (black holes or white holes). Here by contrast we only consider massive
particles, that is, cone singularities along time-like segments.

However we will rely at some point on the analysis made in \cite{colI} of the geometry
near an ``interaction of particles'', that is, 
a vertex of the singular graph. 
Let us briefly recall here a simplified version of the notion of HS-surfaces introduced in \cite{colI}, suited to the purpose
of the present paper, which allows to describe collisions of massive particles.
Let $p$ be a point in $\AdS_3$. The tangent of space $T_p\AdS_3$ is
a copy of Minkowski space $ \RR^{1,2}$. The link $L(p)$ at $p$ is the space of non-oriented geodesic rays
based at $p$; it is naturally identified with the space $\HS^2$ of half-lines in the vector space $\RR^{1,2}$. It admits
a natural decomposition in five subsets:
\begin{itemize}
\item the domains $\HH^2_+$ and $\HH^2_-$ comprising respectively future oriented and past oriented
time-like rays,
\item the domain $\dS^2$ comprising space-like rays,
\item the two circles $\partial\HH^2_+$ and $\partial\HH^2_-$, boundaries of $\HH^2_\pm$ in $\HS^2$.
\end{itemize}

The domains $\HH^2_\pm$ are notoriously Klein models of the hyperbolic plane,
and $\dS^2$ is the Klein model of de Sitter space of dimension $2$. The group $\SO_0(1,2)$,
i.e. the group of of time-orientation preserving and orientation preserving isometries of
$\R^{1,2}$, acts naturally (and projectively) on $\HS^2$, preserving this decomposition.

\begin{defi}
A HS-surface is a topological surface endowed with a $(\SO_0(1,2), \HS^2)$-structure.
\end{defi}

A HS-surface admits a decomposition in hyperbolic and de Sitter regions, delimited by lines or circles
of \textit{photons}, corresponding to the circles $\partial\HH^2_\pm$ in $\HS^2$.

We now define the notion of \textit{singular $\HS$-surface.}
Since here we only consider massive particles, we can restrict the definition given in \cite{colI} and adopt the following
definition:

\begin{defi}
A singular HS-surface is a topological surface $\Sigma$ containing a finite subset $P = \{p_1, ... ,p_k\}$ (the \textit{singular points})
such that the regular part $\Sigma_{reg} = \Sigma \setminus P$ is a HS-surface. Moreover, we require that every singular point $p_i$ admits an open
neighborhood $U_i$ such that $U_i \setminus \{ p_i\}$ lies in 
the hyperbolic region of $\Sigma$, and is isometric to the neighborhood of the singular point in 
$\HH^2_{\theta_i}$ for some $\theta_i$ in $[0, 2\pi]$.
\end{defi}

Given a singular HS-surface $\Sigma$ homeomorphic to the sphere $\mathbb{S}^2$, one can construct a $3$-manifold $e(\Sigma)$ containing a closed
subset $\cL$ (the singular locus) such that $e(\Sigma) \setminus \cL$ is a regular AdS-spacetime, and such that $\cL$ is the union of a single point $p_0$ (the
\textit{collision point}) and singular rays which are massive particles based at $p_0$. More precisely, $\Sigma$ can be interpreted as the space
of geodesic rays starting from $p_0$, every singular point $p_i$ corresponding to a massive particle beginning or finishing at $p_0$.

Mathematically speaking, the singularity along each singular ray in $\cL$ is a cone singularity along a timelike line. If the cone angle is less than $2\pi$ there is a simple way to describe this singularity.
Consider the region of the Anti de Sitter space $U$ bounded by two timelike half-planes that meet
along a time-like geodesic $l$ and that form an angle $\theta$. Then the space obtained by gluing
the faces of $U$ by a rotation around $l$ is the model of a particle of cone angle $\theta$.

Here we require all the masses to be positive, meaning that every $\theta_i$ is less than $2\pi$. According to \cite[Theorem 5.6]{colI}, and due to our restrictions here,
there are only two possibilities:

\begin{itemize}
\item $\Sigma$ has no de Sitter region, i.e. is reduced to its hyperbolic region. If this hyperbolic region is future (i.e. develops into $\HH^+$), then
$e(\Sigma)$ is contained in the future of $p_0$, which can be interpreted as a ``Big Bang'' singularity. If not, then $p_0$ is a ``Big Crunch'' singularity.
\item $\Sigma$ has one future hyperbolic region and a past hyperbolic region, both homeomorphic to the disc, and are connected by a unique
de Sitter region homeomorphic to the annulus. The rays contained in the past of $p_0$ correspond to the elements of
$P$ lying in the future hyperbolic component; they are massive particles colliding at
$p_0$. This collision produces a new set of massive particles which are the singular rays in the future of $p_0$.
\end{itemize}

Figure \ref{fig.collision} 
illustrates the second situation: it represents the collision of two massive particles producing only one
massive particle.

\begin{figure}[ht]
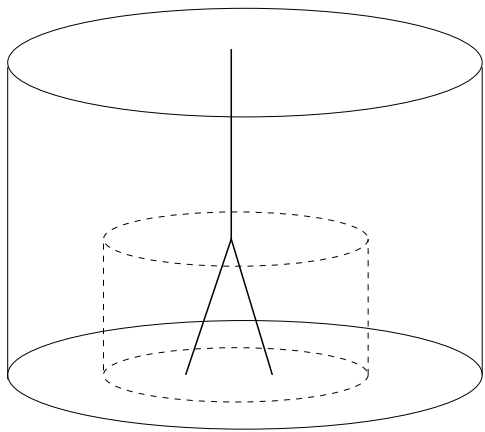

\input interaction.pstex_t
\caption{A collision of two particles.}
\label{fig.collision}
\end{figure}

\begin{remark}
Let $\theta_1$, $\theta_2$ denote the cone angle of the two massive particles in the past, and let $\theta$
be the cone angle of the massive particle in the future. Observe that the holonomy around the future singular point
is a rotation of angle $\theta$, that must be equal to the composition of one rotation of angle $\theta_1$ and a rotation of
angle $\theta_2$ \textit{with distinct centers of rotation.} Hence we have the inequality $\theta < \theta_1 + \theta_2$, which
could at first glance appears as a violation of the conservation of mass. Actually there is no paradox here and this phenomena
is well-known by physicists, the point is that the preserved quantity is the energy-momentum. Here we don't develop further
this kind of consideration, and refer (for example) to \cite[section 3]{matschull}.
\end{remark}

\subsection{Globally  hyperbolic AdS spacetimes with particles}

In \cite{colI} we did a detailed study of the notion of global hyperbolicity in the case of AdS-spacetimes with particle.
We proved in particular (Proposition 6.24) that a globally hyperbolic AdS-spacetime with particles and \textit{no interaction}
(i.e. without collision) admits a maximal globally hyperbolic extension, which is unique up to isometry. As pointed out
in \cite[Remark 6.25]{colI}, this result is completely false if one allows collisions. Indeed, once a collision occurs, there
is no way to predict how many particles will be produced by the collision, and in which direction they will propagate (except
in the special case where only one particle is produced). Hence, Proposition 6.24 in \cite{colI} must be \fb{understood} as a result
\jms{on} the uniqueness of maximal extensions of globally hyperbolic spacetimes, \textit{as long as collisions do not occur.} Such
a spacetime, which can be maximal among spacetimes without collision, may still be extended in a bigger globally hyperbolic
spacetime, but where some collision must happen, \jms{and what happens after the collision is not uniquely determined.}

The point here is that given a AdS spacetime $M$ with interacting particles and a Cauchy surface $S$ in $M$ containing no
collision, there is a unique maximal globally hyperbolic spacetime $M(S)$ with particles but containing no collision, which must coincide
in the neighborhood of $S$ to a neighborhood of $S$ inside $M$. 
This observation will be crucial in the section~\ref{sub:maximal} where we introduce the notion of $m$-spacetime.

\subsection{Admissible AdS spacetimes with interacting particles}

We consider in this paper  AdS spacetimes with collisions of particles, defined as
pairs $(M, T)$ where
\begin{itemize}
\item $M$ is a differentiable manifold and $T$ is a  closed graph embedded in $M$,
\item a smooth AdS metric is defined on $M_{reg}=M\setminus T$,
\item  each edge of $T$ is a massive particle, that is a cone singularity along a time-like curve;
\item each vertex of $T$ is a collision. This means that a neighborhood of each vertex isometrically
embeds in a collision model $e(\Sigma)$ for some HS-surface $\Sigma$. 
\end{itemize}

Notice that we do not consider here the cases where $e(\Sigma)$ is a Big Bang or a Big Crunch. 


Given an AdS spacetime with particles $(M, T)$, we define an isotopy in $(M,T)$ as a
homeomorphism $\phi:M\rightarrow M$ such that there exists a one-parameter
family $(\phi_t)_{t\in [0,1]}$ of homeomorphisms from $M$ to $M$, with
$\phi_0=Id, \phi_1=\phi$, such that each $\phi_t$ sends the singular set of
$M$ to itself. Two domains in $M$ are isotopic if there is an isotopy relative to $T$ sending
one to the other.

We will assume that the cone angle at each edge is in $(0,2\pi)$ (positivity of
the mass) and that the the link of any vertex of $T$ is an $HS$-surface with positive mass
(see Definition 4.2 of \cite{colI}).

We will  also require some natural causal properties.
First we will consider the case where $M$ is topologically the product $S_g\times\mathbb R$, where
$S_g$ is a closed surface of genus $g$. We require that
there is a sequence of embedded closed space-like surfaces
$S_1, S_2,\ldots, S_n$ in $M$ which do not meet the vertices of $T$ such that:
\begin{itemize}
\item $S_{k+1}$ is contained in the future of $S_k$,
\item the region bounded by $S_k$ and $S_{k+1}$ contains exactly one collision,
\item the past of $S_1$ and the future of $S_n$ do not contain any collision point,
\item inextendible causal curves meet every $S_k$ at one point.
\end{itemize}

Since the future and the past of each  point of $T$  is defined, curves ending at $T$
are extendible. Moreover, since it is possible to construct an inextendible causal curve
contained in $T$, $T$ meets each surface $S_k$.

\begin{defi} \label{df:admissible}
\label{def:admissible}
An AdS spacetime with collisions is {\bf admissible} if it satisfies the conditions above.
\end{defi}

\begin{remark}
In general, the sequence $S_k$ is not unique up to isotopies of the pair $(M, T)$.
However $S_1$ and $S_n$ are uniquely
determined up to isotopy by the property that the past of $S_1$ and the future of $S_n$
do not contain collision points.
\end{remark}

\subsection{Maximal spacetimes}
\label{sub:maximal}

Let $M$ be an admissible  spacetime with \jms{collisions}.
Let us consider a space-like surface
 $S_i$ (resp. $S_f$) in $M$
such that the past of $S_i$  and the future of $S_f$
do not contain any collision point.

The past of $S_i$ in $M$ can be thickened to a  maximal globally
hyperbolic spacetime without collision containing $\slice_i$, say $M_i$ .
Analogously the past of $S_f$ can be thickened to a maximal globally hyperbolic
spacetime without collision, say $M_f$. \jms{Note that the notion of maximal globally
hyperbolic spacetime without collision used here is taken from \cite{colI}, that is,
it corresponds to a spacetime which extends up to the point where a collision 
between the particles is bound to happen. The existence of the maximal globally
hyperbolic spacetimes $M_i$ and $M_f$ follow from \cite[Proposition 6.24]{colI}.}

The maximality of $M_i$ shows that
$I^-_M(\slice_i)$ isometrically embeds in $M_i$. In general its image
in $M_i$ is contained in the past of $\slice_i$ in $M_i$, but does not
coincide with it.

We say that $M$ is a $m$-spacetime, if the following conditions hold.
\begin{itemize}
\item $I^-_M(\slice_i)$ isometrically embeds in $M_i$ and coincides with
$I^-_{M_i}(\slice_i)$,
\item $I^+_M(\slice_f)$ isometrically embeds in $M_f$ and coincides with
$I^+_{M_f}(\slice_f)$.
\end{itemize}

\begin{lemma}\label{mm:lem}
Every  admissible spacetime with \jms{collisions} embeds in a $m$-spacetime.
\end{lemma}

\begin{proof}
Let $\slice_i$  and $\slice_f$ be space-like surfaces as above
 and denote by $U_i$ (resp. $U_f$) the past
(resp. the future) of $\slice_i$ (resp. $\slice_f$) in $M$.  Clearly $U_i$
embeds in $M_i$. Let $V_i$ be the past of $U_i$ in $M_i$
(and analogously let $V_f$ be the future of $U_f$ in $M_f$).

Then $V_i$ and $V_f$ can be glued to $M$ by identifying $U_i$ to its
image in $V_i$ and $U_i$ to its image in $U_f$.  The spacetime obtained
in this way, say $M'$, is clearly a $m$-spacetime.
\end{proof}

Let $M$ be an admissible spacetime with \jms{collisions} and $T$ be its singular locus.
We say that $M$ is maximal if every isometric embedding
\[
    (M,T)\rightarrow (M',T')
\]
that restricted to $T$ is a bijection with the singular locus $T'$ of
$M'$, is actually an isometry.

\begin{lemma} \label{lm:m-unique}
Every $m$-spacetime is maximal. Every admissible spacetime $(M,T)$ isometrically
embeds in a unique $m$-spacetime $(M',T')$ such that each vertex of $T'$
is the image of a vertex of $T$.
\end{lemma}

\begin{proof}
We sketch the proof, leaving details to the reader.

Let $(M,T)$ be an admissible spacetime and let $\pi:(M,T)\rightarrow(M_{m},
T_{m})$ be the $m$-extension constructed in Lemma \ref{mm:lem}.

We need to prove that given any isometric embedding
\[
    \iota:(M,T)\rightarrow (M',T')
\]
which restricts to a bijection between $T$ and $T'$,
there is an embedding $\pi':(M',T')\rightarrow (M_{m}, T_{m})$
such that $\pi=\pi'\circ\iota$.

Let $\slice_i,\slice_f$ be space-like surfaces in $M$ as in Lemma
\ref{mm:lem}.  The embedding $\iota$ identifies $\slice_i$ and
$\slice_f$ with disjoint space-like surfaces in $M'$ (that with some
abuse we will still denote by $\slice_i,\slice_f$). Moreover
$\slice_i$ is in the past of $\slice_f$ in $M'$. Since $\iota$ bijectively sends $T$
to $T'$, the past of $\slice_i$ in $M'$ and the future of $\slice_f$ in $M'$
do not contain collision points and are globally hyperbolic domains.

Now the closure of the domains $\Omega=I^+_M(\slice_i)\cap
I^-_M(\slice_f)$ and $\Omega'=I^+_{M'}(\slice_i)\cap I^-_{M'}(\slice_f)$ are
both homeomorphic to $\slice\times[0,1]$ and $\iota$ sends
$\overline\Omega$ to $\overline\Omega'$ and $\partial\Omega$ onto
$\partial\Omega'$. A standard topological argument shows that
$\iota(\Omega)=\Omega'$.

Finally if $M_i$ (resp. $M_f$) denotes the GHMC spacetime without collisions containing
$\slice_i$ (resp. $\slice_f$), we have that $I^-_{M'}(\slice_i)$
(resp. $I^+_{M'}(\slice_f)$) embeds in $I^-_{M_i}(\slice_i)$
(resp. $I^+_{M_f}(\slice_f)$).

Thus we can construct the map $\pi':M'\rightarrow M_{m}$ in such a way that
\begin{itemize}
\item on $I_{M'}^+(\slice_i)\cap I_{M'}^-(\slice_f)$ we have $\pi'=\iota^{-1}$;
\item on $I_{M'}^-(\slice_i)$ it coincides with the embedding in
  $I^-_{M_i}(\slice_i)$;
\item on $I_{M'}^+(\slice_f)$ it coincides with the embedding in
  $I^+_{M_f}(\slice_f)$.
\end{itemize}
\end{proof}

\subsection{The deformation space of maximal spacetimes}

In this section we introduce the space of deformations of a spacetime with collisions.
Let us fix $g$, an oriented graph $T$ in $\slice_g\times\mathbb R$
such that non-compact edges are properly embedded, and a family
of numbers $\theta=(\theta_e)_{e\in T_1}$, where $T_1$ is the set of
edges of $T$.

We consider the space $\tilde\Omega$ of
maximal admissible AdS-structures
on $\slice_g\times\mathbb R$ with
singular locus $T$ such that
\begin{itemize}
\item every edge $e$  is a particle of angle $\theta_e$;
\item the orientation of $e$ agrees with the time-orientation induced by
 the AdS structure;
\item every vertex of $T$ admits a neighborhood in $M$ which embeds in
$e(\Sigma)$ for some HS-surface $\Sigma$ 
\end{itemize}


We denote by $\struct$ the   quotient space: an element of
$\struct$ is  a singular metric with the properties
described above, up to isotopies relative to $T$.



There is a natural forgetful map from $\struct$ to the set of
AdS structures on $(\slice_g\times\mathbb R)\setminus T$ up to
isotopy.
Proposition \ref{ext:prop} ensures that  this map is injective, so $\struct$ can be identified
to a subset of the space of anti-de Sitter structures on $\slice\times\mathbb
R\setminus T$.  Thus $\struct$ inherits from this
structure space a natural topology (see \cite[Section 1.5]{canary-epstein-green} for a
discussion on the topology of the space of $(G,X)$-structures on a
fixed manifold).

\begin{prop}\label{ext:prop}
Let $\mu,\nu$ be two singular metrics on $\slice\times\mathbb R$
with singular locus equal to $T$.

Then any isometry
\[
   \psi:(\slice\times\mathbb R\setminus
   T,\mu)\rightarrow(\slice\times\mathbb R\setminus T,\nu)
\]
extends to an isometry $\bar\psi:(\slice\times\mathbb
R,\mu)\rightarrow(\slice\times\mathbb R,\nu)$.
\end{prop}

\begin{proof}


Let us take $p\in T$.  Consider some small space-like $\mu$-geodesic
$c:[0,1]\rightarrow\slice\times\mathbb R$ such that $c(1)=p$ and
$c([0,1))\cap T=\emptyset$. If $c$ is small enough, we can find two
  points $r_-$ and $r_+$ in $\slice\times\mathbb R\setminus T$ such
  that $c[0,1)\subset I^+_\mu(r_-)\cap I^-_\mu(r_+)$.

Now let us consider the space-like $\nu$-geodesic path $c'(t)=\psi(c(t))$ defined
in $[0,1)$.  Notice that $c'$ is an inextendible geodesic path in
  $\slice\times\mathbb R\setminus T$.

We know that $c'$ is contained in $I^+_\nu(\psi(r_-))\cap
I^-_\nu(\psi(r_+))$. Thus if $\slice_\pm$ is a space-like surface through
$r_\pm$, we have that $c'\subset I^+_\nu(\slice_-)\cap
I^-_\nu(\slice_+)$ that is a compact region in $\slice\times\mathbb
R$.
Thus $c'(t)$ has accumulation points as $t\rightarrow 1$. All these accumulation points lie in $T$; if
there are two different accumulation points, then there is a segment in $T$ accumulated by $c'$.
This is a contradiction since $c'$ is space-like whereas $T$ is time-like.

Hence, $c'(t)$ converges to some point in $T$ as $t\rightarrow 1$.
We define $\hat c=\lim_{t\rightarrow 1}c'(1)$.

To prove that $\psi$ can be extended on $T$ we have to check that
$\hat c$ only depends on the endpoint $p$ of $c$. In other words, if
$d$ is another space-like geodesic arc ending at $p$, we have to prove that $\hat
c$ is equal to $\hat d=\lim_{t\rightarrow1}\psi\circ d(t)$.  By a
standard connectedness argument, there is no loss of generality if we
assume that $d$ is  close to $c$. In particular we may assume
that there exists the space-like geodesic triangle $\Delta$ with vertices $c(0),
d(0), p$.

Consider now the $\mu$-geodesic segment in $\Delta$, say $I_t$, with
endpoints $c(t)$ and $d(t)$.  The image $\psi(I_t)$ is a
$\nu$-geodesic segment contained in $\slice\times\mathbb R\setminus T$.
Arguing as above, we can prove that
all these segments $(\psi(I_t))_{t\in[0,1)}$ are contained in some
compact region of $\slice\times\mathbb R$.  Thus either they converge
to a point (that is the case $\hat c=\hat d$), or they converge to
some geodesic path in $T$ with endpoints $\hat d$ and $\hat c$.

On the other hand, the $\nu$-length of $\psi(I_t)$ goes to zero as
$t\rightarrow 1$.  Thus either $\psi(I_t)$ converges to a point or it
converges to a lightlike path. Since $T$ does not contain any
lightlike geodesic it follows that $\psi(I_t)$ must converge to a
point.  Thus $\hat d=\hat c$.

Finally we can define
\[
   \psi(p)=\hat c
\]
where $c$ is any space-like $\mu$-geodesic segment with endpoint equal to $p$.

Let us prove now that this extension is continuous.
For a sequence of points $p_n$ converging to $p\in T$ we have to
check that $\psi(p_n)\rightarrow\psi(p)$.
We can reduce to consider  two cases:
\begin{itemize}
\item $(p_n)$ is contained in $T$,
\item $(p_n)$ is contained in the complement of $T$.
\end{itemize}
In the former case we consider a point $q$ in the complement of $T$ close to $p$.
Let us consider the $\mu$-geodesic segment $c$ joining $q$ to $p$ and
the segments $c_n$ joining $q$ to $p_n$. Clearly for every $t\in (0,1]$ the points
$c_n(t)$ and $c(t)$
are time related and their Lorentzian distance converges to the
distance between $p_n$ and $p$ as $t\rightarrow 1$. On the other hand, since
$\psi(c_n(t))$, $\psi(c(t))$ converges respectively to $\psi(p_n)$, $\psi(p)$ as
$t\rightarrow 1$, we can conclude that
\[
    d(\psi(p_n),\psi(p))=d(p_n,p)
\]
where $d$ denotes the Lorentzian distance along $T$. This equation implies that
$\psi(p_n)\rightarrow\psi(p)$ as $n\rightarrow+\infty$.

Let us suppose now that the points $p_n$ are contained in the complement of $T$.
We can take $r_+$ and $r_-$ such that $p_n\in I^+_\mu(r_-)\cap I^-_\mu(r_+)$ for $n\geq n_0$.
Thus the same argument used to define $\psi(p)$ shows  that $(\psi(p_n)) $ is contained in
some compact subset of $\slice\times\mathbb R$. To conclude it is sufficient to prove that
if $\psi(p_n)\rightarrow x$ then $x=\psi(p)$.
Clearly $x\in
T$. Moreover either $x$ coincides with $\psi(p)$ or there is a
piece-wise geodesic segment in $T$ connecting $x$ to $\psi(p)$. Since the length of this
geodesic should be equal to the limit of $d(p_n,p)$, that is $0$, we
conclude that $x=\psi(p)$.

Eventually we have  to check that the map $\psi$ is an isometry at $p$. Let us note
that $\psi$ induces a map
\[
\psi^\#:\link_p\rightarrow \link'_{\psi(p)}~,
\]
where $\link_p$ and $\link'_{\psi(p)}$ are respectively the link of $p$ with
respect to $\mu$ and the link of $\psi(p)$ with respect to $\nu$.
Simply, if $c$ is the tangent vector of a geodesic arc $c$ at $p$, we
define $\psi^\#(v)=w$ where $w$ is the tangent vector to the arc
$\psi\circ c$ at $\psi(p)$.  Notice that $\psi$ is an isometry around
$p$ if and only if $\psi^\#$ is a $HS$-isomorphism.

Clearly $\psi^\#$ is bijective and is an isomorphism of $HS$ surfaces
outside the singular locus. On the other hand,
the singularities
are contained in the hyperbolic regions, which are the metrics completions of their
regular parts. Hence the bijection $\psi^\#$ is an extension of the isometry between
the regular parts, therefore a $HS$-isomorphism.
\end{proof}

\section{The holonomy map on the space of admissible spacetimes}
\label{sc:holonomy}

\subsection{Holonomies of singular AdS-spacetimes around the singular locus}
\label{ssc:holonomies}

Let $(M, T)$ be a an admissible AdS structure on $\slice_g\times\mathbb R$. Recall
that $M_{reg}$ is the regular part $M \setminus T$. We consider the holonomy
\[
   \hol:\pi_1(M_{reg})\rightarrow SO(2,2)~.
\]
Fix a collision point, $p$,  of
$M$, and let $\Sigma_p$ be the link of the point $p$.
Notice that the inclusion map $\iota:\Sigma_p\rightarrow M$ produces
an inclusion of groups well-defined up to conjugation
\[
  \iota_\star:\pi_1(\Sigma_{p,reg})\rightarrow\pi_1(M_{reg})\, ,
\]
 where $\Sigma_{p,reg}$ is the
regular part of $\Sigma_p$.

In this section we will investigate the behavior of the restriction
of $\hol$ to $\pi_1(\Sigma_{p,reg})$. 

\begin{lemma}\label{lm:holadm}
If $\gamma$ is a meridian loop in $\pi_1(M_{reg})$ around a particle $e$
\jms{--- a loop going once around $e$ ---}
then $\hol(\gamma)$ is a rotation of angle $\theta_e$ around a time-like
geodesic of $AdS_3$.

If $p$ is a collision point, then there is $x_0\in AdS_3$ which is fixed
by $\hol(\gamma)$ for any $\gamma\in \pi_1(\Sigma_{p,reg})$.
\end{lemma}

\begin{proof}
In the first case a regular neighborhood of $\gamma$ embeds in the local
model of the particle whose holonomy is a rotation of angle $\theta_e$.

In the second case, notice that a neighborhood of the point $p$ is isometric
to the AdS cone of the $HS$ surface $\Sigma_p$. So in order to prove the statement
it is sufficient to check that the holonomy of the AdS cone of a $HS$ surface fixes
a point in $AdS_3$.

Now fix any point $x_0\in AdS_3$ and identify $T_{x_{0}}AdS_3$ with the
Minkowski space $\mathbb R^{2,1}$. In this way the $HS$-sphere is identified with the
space of geodesic rays starting from $\tilde x_0$.
If $c_0:\tilde\Sigma_p\rightarrow HS_2$ is the developing map of $\Sigma_p$, then the
developing map of the AdS  cone of $\Sigma_p$ is the map
\[
dev:\tilde\Sigma_p\times\mathbb R\rightarrow AdS_3
\]
such that if $x$ is not a photon of $\tilde\Sigma_p$ then
$dev(x,t)=c_0(x)[t]$, where $c_0(x)[\bullet]$ is the arc-length parameterization
of the segment $c_0(x)$.

In particular, if $\gamma$ is an element of $\pi_1(\Sigma_{p,reg})$ then the holonomy
of the AdS cone around $\gamma$ is the transformation of $SO(2,2)$ which fixes
 $x_0$ such that its differential at $x_0$ coincides with the holonomy of
 $\Sigma_p$ (as $HS$-surface) around $\gamma$.
 \end{proof}

\subsection{The holonomy map}

After the preliminary material in the previous section, we now turn to the statement
and proof of Theorem \ref{struct:thm}. We first define the notion of ``admissible''
holonomy representations.

\begin{defi}\label{df:adrep}
We say that a representation
\[
\hol:\pi_1(M_{reg})\rightarrow SO(2,2)
\]
is admissible if
\begin{itemize}
\item for any meridian loop $\gamma$ in $M_{reg}$ around an edge $e$ of $T$
$\hol(\gamma)$ is a rotation around a time-like geodesic in $AdS_3$ of angle $\theta_e$,
\item for any vertex $p\in T$, there is a point $x_0\in AdS_3$ fixed by
$\hol(\gamma)$ for any $\gamma\in\pi_1(\Sigma_{p,reg})$.
\end{itemize}
\end{defi}

We denote by $\mathcal R(g, T, \theta)$ the space of admissible representations up to
conjugacy.

By Lemma \ref{lm:holadm} the holonomy of any structure of $\struct$ lies
in $\mathcal R(g, T, \theta)$. We prove now that $\struct$ is locally
homeomorphic to $\mathcal R(g, T, \theta)$.

\begin{theorem} \label{struct:thm}
The holonomy map
\[
\struct\rightarrow \mathcal R(g, T, \theta)
\]
is a local homeomorphism.
\end{theorem}

To prove this proposition we will use the following well-known fact
about $(G,X)$-structures on compact manifolds with boundary, see
\fb{Theorem 1.7.1} in \cite{canary-epstein-green}.
Let us recall that a collar of a manifold $N$ with boundary is a
neighborhood
of the boundary homeomorphic to $\partial N\times[0,1)$.

\begin{lemma}\label{ep:lem}
Let $N$ be a smooth compact manifold with boundary and let $N'\subset
N$ be a submanifold such that $N\setminus N'$ is a collar of $N$.
\begin{itemize}
\item
Given a $(G,X)$-structure $M$ on $N$ let $hol(M):\pi_1(N)\rightarrow G$
be the corresponding holonomy (that is defined up to conjugacy).
Then, the holonomy map from the space of $(G,X)$-structures on $N$ to
  the space of representations of $\pi_1(N)$ into $G$ (up to conjugacy)
\[
    M\mapsto hol(M)
\]
is an open map.
\item
Let $M_0$ be a $(G,X)$-structure on $N$ and denote by $M_0'$ the
restriction of $M_0$ to $N'$.
There is a neighborhood $\mathcal U$ of $M_0$ in the set of
$(G,X)$-structures on $N$ and a neighbourhood $\mathcal V$ of $M'_0$
in the set of $(G,X)$-structures on $N'$ such that
\begin{enumerate}
\item If $M\in\mathcal U$ and $M'\in\mathcal V$ share the same holonomy,
there is an embedding as $(G,X)$-manifolds
\[
    M'\hookrightarrow M.
\]
homotopic to the inclusion $N'\hookrightarrow N$.
\item For every $M'\in\mathcal V$ there is $M$ in $\mathcal U$ such that $hol(M)=hol (M')$.
\end{enumerate}
\end{itemize}
\end{lemma}

First we prove Theorem \ref{struct:thm} assuming just one collision
in $M$.  Let $p_0$ be the collision point of $M$ and $\link_0$ be the link
of $p_0$ in $M$ (that is a HS-surface). Denote by
$\link_{0,reg}$ the regular part of $\link_0$
and by $G_0<\pi_1(\slice\times\mathbb R\setminus T)$ the fundamental group
of $\link_{0,reg}$.

Given any representation $h\in\mathcal R$, let us denote by $x_0$ the point
fixed by $h(G_0)$ and by  $Lh:G_0\rightarrow SO(2,1)=SO(T_{x_0}AdS_3)$
the action of $G_0$ at the tangent space of $x_0$.
Notice that, identifying $SO(T_{x_0}AdS_3)$ with $SO(2,1)$,
the conjugacy class of $Lh$ only depend of the conjugacy class of $h$. Moreover
the map sending $h$ to $Lh$ is a continuous map between $\mathcal R$
and the space of conjugacy classes of representations of $G_0$ into
$SO(2,1)$.

\begin{lemma}\label{bo:lem}
There is a neighborhood $\mathcal U_0$ of $\link_0$ in the space of
HS-surfaces homeomorphic to $\link_0$ such that the holonomy map on
$\mathcal U_0$ is injective.

Moreover, there is a neighbourhood $\mathcal V$ of $h$ in $\mathcal R(g, T, \theta)$ such
that for every $h'\in\mathcal V$ there is an HS-surface in
$\mathcal U_0$, say $\link(h')$, such that the holonomy of $\link(h')$ is conjugate to
$Lh':G_0\rightarrow SO(2,1)$.
\end{lemma}

\begin{proof}
Around each cone  point $q_i$ of $\link_0$ take small disks
\[
   \Delta_1(i)\supset\Delta_2(i)
\]
Let now $\link,\link'$ be two HS-surfaces close to $\link_0$ sharing the same
holonomy.  By Lemma \ref{ep:lem}, up to choosing $\mathcal U_0$
sufficiently small, there is an isometric embedding
\[
  f:(\link\setminus\bigcup\Delta_2(i))\rightarrow \link'.
\]
 Moreover, $\Delta_1(i)$ equipped with the structure induced by $\link$
 embeds in $\link'$ (this because the holonomy locally determines the
 HS-structure near the singular points of HS-surfaces). It is not difficult to see that
 such an inclusion coincides with $f$ on
 $\Delta_1(i)\setminus\Delta_2(i)$ (basically this depends on the fact
 that an isometry of a hyperbolic annulus into a disk containing a
 cone point is unique up to rotations).  Thus gluing those maps we
 obtain an isometry between $\link$ and $\link'$.

To prove the last part of the statement, let us consider for each cone
point a smaller disk $\Delta_3(i)\subset\Delta_2(i)$.  Let
$U=\link_0\setminus\bigcup\Delta_3(i)$. Clearly we can find a
neighbourhood $\mathcal V$ of $h$ such that if $h'\in \mathcal V$ then
there is a structure $U'$ on $U$ close to the original one with holonomy
$Lh'$. On the other hand it is clear that there exists a structure, say
$\Delta'_1(i)$, on $\Delta_1(i)$ with cone singularity with holonomy
given by $Lh'$ and close to the original structure.  By Lemma
\ref{ep:lem}, if $h'$ is sufficiently close to $h$, then
$\Delta_2(i)\setminus\Delta_3(i)$ equipped with the structure given by
$U'$ embeds in $\Delta'_1(i)$. Moreover $\partial \Delta_2(i)$ bounds
in $\Delta'_1(i)$ a disk $\Delta(i)$ containing the cone point. Thus
we can glue the $\Delta_1(i)$ to $U'$ to obtain the HS-surface with
holonomy $Lh'$.
\end{proof}

Let $C(h')$ be the AdS cone on $\link(h')$.
By construction, the holonomy of $C(h')$
is conjugated to $h'|_{G_0}$.

Consider now two space-like surfaces $\slice_1,\slice_2$ in $M$
orthogonal to the singular locus that are disjoint and such that
$p_0\in I^+(\slice_1)\cap I^-(\slice_2)$.  Let $M_0= I^+(\slice_1)\cap
I^-(\slice_2)$. Clearly $\slice_1$ is the past boundary of $M_0$ and
$\slice_2$ is its future boundary.

Take the neighborhood $\mathcal V$ of $h$ in $\mathcal R$ given by
Lemma \ref{bo:lem} and, for $h'\in\mathcal V$, consider the AdS cone
$C(h')$ constructed above.  The following is a simple application of
Lemma~\ref{ep:lem}, we leave the proof to the reader.

\begin{cor}\label{bo:cr}
Let $N_0$ be the AdS-manifold with boundary obtained by cutting a
regular neighborhood of $T$ from $M_0$,
and let $U$ be a collar of $\partial N_0$ in $N_0$.
If $N'$ is a slight deformation of the AdS-structure on $N_0$ with holonomy
$h'$ then $U$, with the AdS-structure induced by $N'$,
embeds in $C(h')$.
\end{cor}


Now up to shrinking $\mathcal V$ we may suppose that:
\begin{itemize}
\item For any $h'\in\mathcal V$, there is a deformation of the AdS structure on $N_0$,
say $N_0(h')$, with holonomy $h'$.
\item If $U(h')$ denotes the AdS structure induced by $N_0(h')$ on $U$, then
$U(h')$ isometrically embeds in $C(h')$.
\item
The image of the boundary of $N_0(h')$ through this embedding is the frontier
of a regular neighborhood $B$ of the singular locus in $C(h')$.
\item
The image of $U(h')$ is disjoint from $B$
\end{itemize}

The spacetime obtained by gluing $B$ to $N_0(h')$, by identifying
the boundary of $N_0(h')$ with the frontier of $B$,
is a spacetime with \jms{collisions} with holonomy $h'$.
Its maximal extension, say $M(h')$, is a $m$-spacetime
with holonomy $h'$.

To conclude we have to prove that if $h'$ is sufficiently close to
$h$, then $M(h')$ is unique in a neighborhood of $M_0$.

In fact, it is sufficient to show that
any given $m$-spacetime $M'$ with holonomy $h'$ close to $M$ must contain
a spacetime $M'_0\subset M'$ containing the collisions which embeds isometrically in $M(h')$.
We can assume $M'_0$ close to  $M_0$ (this
precisely means that $M'_0$ is obtained by deforming slightly the metric
on $M_0$).

Take small neighborhoods $U_2\subset U_1$ of the singular locus in $M'_0$.
By Lemma \ref{ep:lem} $M'_0\setminus U_2$ embeds in $M(h')$.
By the uniqueness of the $HS$-surface with holonomy $h'$, $U_1$ embeds
in $M(h')$ as well.

It is not difficult to check that there exists a unique
isometric embedding $U_1\cap U_2\rightarrow M(h')$, so
the embeddings $U_1\hookrightarrow M(h')$ and $M'_0\setminus
U_2\hookrightarrow M(h')$ coincide on the intersection. So they can be combined 
to an embedding $M'_0\hookrightarrow M(h')$.

This concludes the proof of Theorem \ref{struct:thm} when only
one interaction occurs.  The following lemma allows to conclude in the
general case by an inductive argument.

\begin{lemma}
Let $\slice$ be a space-like surface of $M$, and
let $M_-$, $M_+$ be the past and the future of $\slice$ in $M$.
Suppose that for a small deformation $h'$ of the holonomy $h$ of $M$
there are two spacetimes with collisions $M_-'\cong M_-$ and
$M_+'\cong M_+$ such that the holonomy
of $M_\pm'$ is equal to $h'|_{\pi_1(M_\pm)}$. Then there is a
spacetime $M'$ close to $M$ containing both $M_-'$ and $M_+'$.
\end{lemma}

\begin{proof}
Let  $N(h)$ denote the maximal GH structure with
particles on $\slice\times\mathbb R$ whose  holonomy is $h|_{\pi_1(\slice_{\fb{reg}})}$.
There is a neighborhood of $\slice$ in $M$ which embeds in $N(h)$.
We can suppose that $\slice\subset M$ is sent to $\slice\times\{0\}$
through this embedding.

Now let $U_\pm$ be a collar of $\slice$ in $M_\pm$ such that the image of
$U_-$ in $N(h)$ is $\slice\times [-\epsilon, 0]$ and the image of
$U_+$ is $\slice\times[0,\epsilon]$ for some $\epsilon>0$.

 If $h'$ is sufficiently close to $h$, then there is an isometric
 embedding of $U_\pm$ (considered as subset of $M_\pm'$) into $N(h')$
 \[
    i_\pm: U_\pm\hookrightarrow N(h')
 \]
 such that the image of $U_-$ is contained in $\slice\times
 [-2\epsilon, \epsilon/3]$ and contains $\slice\times\{-\epsilon/2\}$,
 and that the image of $U_+$ is contained in
 $\slice\times[-\epsilon/3,2\epsilon]$ and contains
 $\slice\times\{\epsilon/2\}$. Thus we can glue $M'_\pm $ and
 $\slice\times [-\epsilon/2,\epsilon/2]$ by identifying $p\in
 U_\pm\cap i_\pm^{-1}(\slice\times[-\epsilon/2,\epsilon/2])$ with its
 image.  The spacetime we obtain, say $M'$, clearly contains $M'_-$ and
 $M'_+$.
\end{proof}

\begin{remark}
To prove that there is a unique $m$-spacetime in a \fb{neighborhood} of $M$
with holonomy $h'$, we again use an inductive argument. Suppose we can
find in any small neighborhood of $M$ two $m$-spacetimes $M'$ and
$M''$ with holonomy $h'$. We fix a space-like surface
 $\slice$ in $M$ such that both the future and the past of
$\slice$, say $M_\pm$, contain some collision points.  Let $U\subset
V$ be regular \fb{neighborhoods} of $\slice$ in $M$ with space-like
boundaries.  We can consider collars $U'\subset V'$ in $M'$ and
$U''\subset V''$ in $M''$ such that
\begin{itemize}
\item $U'\cup U''$ and $V'\cup V''$ are close to $U$ and $V$ respectively,
\item they do not contain any collision,
\item they have space-like boundary.
\end{itemize}
Applying the inductive hypothesis on the connected regions of the
complement of $U'$ in $M'$ and $U''$ in $M'$ we have that for $h'$
sufficiently close to $h$ there is an isometric embedding
\[
    \psi:M'\setminus U'\rightarrow M''
\]
such that $\psi(\partial U')$ is contained in $V''$.

Now consider the isometric embeddings
\[
   u':V'\rightarrow N(h')\qquad u'':V''\rightarrow N(h')
 \]
 \fb{where $N(h')$ is the GHMC structure on $\slice\times\mathbb R$
 whose holonomy is $h'|_{\pi_1(\slice_{reg})}$. Notice}
  that the maps $u'$ and $u''\circ\psi$ provide two isometric embeddings
 \[
     V'\setminus U'\rightarrow N(h')
 \]
 so they must coincide (we are using the fact that the inclusion of a
 GH spacetime with particles in its maximal extension is uniquely
 determined).

 Finally we can extend $\psi$ on the whole $M'$ by setting on $V'$
 \[
    \psi= (u'')^{-1}\circ u'\,.
\]
\end{remark}

\section{The left and right metrics on space-like slices of good spacetimes}
\label{sc:leftright}

The main goal of this section is 
to construct, for each space-like slice containing no particle collision, two
hyperbolic metrics with cone singularities on a surface. It is the sequence
of those pairs of hyperbolic metrics (or, more precisely, the graph of those
pairs of hyperbolic metrics) which provide a complete description of a
spacetime with interacting particles, as seen in Theorem \ref{tm:main}.

\subsection{The left and right connections} \label{ssc:leftright}

The constructions of the left and right hyperbolic metrics, below, can be understood
in a fairly simple manner through two flat linear connections
on the tangent bundle of  an AdS 3-manifold. In this first part
we consider an AdS manifold $M$, which could for instance be the regular
part of an AdS manifold with particles.

\begin{defi}\label{df:Dleft}
Let $M$ be an AdS manifold and $\nabla$ be its Levi-Civita connection.
On $M$ we consider two linear connections defined by
\[
D^l_vu=\nabla_vu+u\times v~,~~D^r_vu=\nabla_vu-u\times v~,
\]
where $\times$ is
the cross-product in $AdS_3$ --- it can be defined by $(v\times y)^*=*(v^*\wedge
y^*)$, where $v^*$ is the 1-form dual to $v$ for the AdS metric and $*$ is the
Hodge star operator.
\end{defi}

\begin{lemma}
$D^l$ and $D^r$ are flat connections compatible with the AdS-metric.
\end{lemma}

\begin{proof}
The fact that $D^l$ and $D^r$ are compatible with the metric easily follows from
the property of the cross-product.

\fb{
Since the cross product is flat with respect to the Levi-Civita connection,
there is a simple relation between 
 the curvature $R^l$ of $D^l$ to the curvature $R$ of $\nabla$,
 that can be proved by a direct computation.
In fact we have
\[
  R^l(v,w)u=R(v,w)u+v\times(w\times u)-w\times(v\times u)~.
\]

A basic point is that the Riemann curvature tensor of a
Lorentzian space form of constant curvature $K$  
can be easily expressed in terms of the cross product.
Indeed if $v,w,u$ are tangent vectors in $M$ we have
\begin{equation}\label{eq:spfrm}
   R(v,w)u=K(v\times w)\times u~.
\end{equation}
So we get
\[
  R^l(v,w)u=u\times(v\times w) + v\times(w\times u) + w\times(u\times v)=0
\]
where the last identity holds for the Jacobi identity for the cross product.
}
\end{proof}

\fb{
\begin{remark}
For a $3$-dimensional Riemannian space form, formula (\ref{eq:spfrm}) holds
with the opposite sign. For this reason the  construction above applied to
the Riemannian setting produces two flat connections on the unit tangent
bundle of a 3-dimensional spherical manifold.

This phenomenon is closely related to the fact that the isometry group
of the three dimensional sphere 
$S^3$ (as well as the isometry group of $AdS_3$) has a natural product structure.
\end{remark}
}

\begin{defi} \label{df:unittmlk}
We call $T^{1,t}M$ the bundle of positively directed unit
time-like vectors on $M$.
\end{defi}

Notice that if $V$ is a unit time-like vector field, then
$D^l_xV$ and $D^r_xV$ are orthogonal to $V$ at any point.
In particular they belong to the tangent space of $T^{1,t}M$.

In this section we want to relate the holonomies of connections $D^l$
and $D^r$ --- that are representations $\pi_1(M)\rightarrow SO_0(2,1)$ ---
to the holonomy of the AdS-structure on $M$, that is a representation
$\pi_1(M)\rightarrow SO(2,2)$.

First we prove that the holonomy of the model space $AdS_3$ is trivial.

\begin{lemma}\label{lm:fl}
The holonomy of $D^l$ and $D^r$ on $AdS_3$ is trivial.
\end{lemma}

\begin{proof}
Since the fundamental group of $AdS_3$ is generated by
the geodesic curve $\gamma(t)=(\cos t,\sin t,0,0)$, it is sufficient to
compute the linear maps
\[
  h^l(\gamma): T_{\gamma(0)}AdS_3\ni v\mapsto V^l(2\pi)\in T_{\gamma(0)}AdS_3
\]
where $V^l(t)$ is the $D^l$-parallel field along $\gamma$ with initial condition
$V^l(0)=v$.

If $v=\dot\gamma(0)$, it is easy to see that $V^l(t)=\dot\gamma(t)$, so
$h^l(\gamma)(v)=v$.

If $v$ is orthogonal to $\dot\gamma$, denote by $V$ the $\nabla$-parallel
field extending $v$. Then we easily see that
\[
          V^l(t)=\cos (t) V(t)\ -\ \sin(t) V(t)\times\dot\gamma~.
\]
Since $V(2\pi)=v$, we obtain that $h^l(\gamma)(v)=v$.

By linearity we conclude that $h^l(\gamma)$ is the identity.
Similarly we can prove that $h^r(\gamma)$ is  the identity.
\end{proof}

Let us fix a base point $x_0\in AdS_3$ and consider the maps
\[
  \tau^l(x), \tau^r(x):T_{x_0}AdS_3\rightarrow T_{x}AdS_3
\]
obtained by using parallel transport with respect to $D^l$ and $D^r$
along any curve joining $x_0$ to $x$. By Lemma \ref{lm:fl}, these maps are
well-defined.

We identify once for all $T_{x_0}AdS_3$ with Minkowski space, and
$O(T_{x_0}AdS_3)$ with $O(2,1)$.

\fb{
Given any isometry $g\in SO(2,2)$, we can consider the linear
transformations of $T_{x_0}M$ obtained by composing the 
differential map $dg(x_0):T_{x_0}AdS_3\rightarrow T_{g(x_0)}AdS_3$
by the inverse of the parallel
transports $\tau_l(g(x_0)),\tau_r(g(x_0)):T_{x_0}AdS_3\rightarrow
T_{g(x_0)}AdS_3$. Namely
\[
\begin{array}{l}
g_l=\tau_l(g(x_0))^{-1}\circ dg(x_0)~, ~~
g_r=\tau_r(g(x_0))^{-1}\circ dg(x_0)~.
\end{array}
\]
}
Notice that $g_l$ and $g_r$ are elements of $SO_0(2,1)$.
\begin{lemma}
The map
\begin{equation}\label{eq:can}
I:SO_0(2,2)\ni g\mapsto (g_l,g_r)\in SO_0(2,1)\times SO_0(2,1)
\end{equation}
is a surjective homomorphism and its kernel is the 
$\mathbb Z/2\mathbb Z$-subgroup generated by the antipodal map.
\end{lemma}
\begin{proof}
Notice that $\tau_l(hg(x_0))=\tau'_l\circ\tau_l(h(x_0))$ where
$\tau'_l$ is the parallel transport $T_{h(x_0)}M\rightarrow
T_{hg(x_0)}M$. Since $D^l$ is preserved by isometries of $AdS_3$ we
deduce that $\tau'_l=dh(g(x_0))\circ\tau_l(g(x_0))\circ
(dh(x_0))^{-1}$.  From these formulas we easily get that
$I$ is a homomorphism.

Given a Killing vector field $X\in\mathfrak{so}(2,2)$ we have that
\[
    dI_{id}(X)=(D^lX(x_0), D^rX(x_0))~.
\]
(Notice that $D^lX$ and $D^rX$ are skew-symmetric operators of $T_{x_0}M$.)
This formula easily shows that $dI_{id}$ is an isomorphism.
We conclude that $I$ is a covering map. Since the center of
$SO_0(2,1)$ is trivial, $\ker I$ is the center of $SO_0(2,2)$, that
is, the group generated by the antipodal map.
\end{proof}

\begin{remark}
Mess \cite{mess} described this map $\SO_0(2,2)\rightarrow
PSL(2,\mathbb R)\times PSL(2,\mathbb R)$ in a different way, using the
double ruling of the projective quadric $C=\{[x]\in \mathbb{P}(\mathbb
R^{2,2})|\langle x,x\rangle=0\}$.
\end{remark}

\begin{lemma}\label{lm:stab}
Through the identification between $SO_0(2,2)$ with $SO_0(2,1)\times
SO_0(2,1)$, the stabilizer in $SO_0(2,2)$ of a point $x\in AdS_3$
corresponds to a subgroup of $SO_0(2,1)\times SO_0(2,1)$ conjugated to
the diagonal subgroup.
\end{lemma}

\begin{proof}
It is sufficient to prove the statement in the case where $x=x_0$.
In that case it is clear by definition that if $g$ fixes $x_0$, then
$g_l=g_r=dg(x_0)^{-1}$. So the stabilizer of $x_0$ is contained
in the diagonal subgroup. Since those groups have the same dimension,
they must coincide.
\end{proof}

\fb{
Let us fix $p_0\in M$. For any loop $\gamma$ centered at $p_0$
let us denote by $\hol_l(\gamma), \hol_r(\gamma)\in SO_0(T_{x_0}M)\cong SO_0(2,1)$ the holonomy along 
$\gamma^{-1}$ with respect to $D^l$ and $D^r$. The reason why we consider the parallel transport along the
 inverse of $\gamma$ is that in this way 
$\hol_\bullet(\gamma\delta)=\hol_\bullet(\gamma)\hol_\bullet(\delta)$.

Since $D^l$ and $D^r$ are flat,  $\hol_l(\gamma),\hol_r(\gamma)$  only depend on the homotopy
class of $\gamma$. In particular two holonomy representations
 $\hol_l,\hol_r:\pi_1(M)\rightarrow SO_0(2,1)$ are associated to $D^l$ 
and $D^r$. Though the construction depends on the choice of a point $p_0$,
those representations are well-defined up to conjugation.
}


\begin{lemma} \label{lm:rho}
Up to conjugation we have that
\[
  I\circ \hol=(\hol_l,\hol_r)~.
\]
\end{lemma}

\begin{proof}
Let us fix a universal covering map $\pi:\tilde M\rightarrow M$,
a base point $p_0\in M$, and a point $\tilde p_0\in\pi^{-1}(p_0)$.
Without loss of generality we may suppose that the developing map
sends $\tilde p_0$ to $x_0$.

Let $\gamma:[0,1]\rightarrow M$ be a closed loop in $M$ such that
$\gamma(0)=\gamma(1)=p_0$.  Consider the lift $\gammab$ of $\gamma$ to
$\tilde M$ with starting point $\tilde p_0$ and denote by
$L=L(\gamma)$ the covering automorphism such that $L(\tilde
p_0)=\gammab(1)$.  Since $\pi\circ L=\pi$ we have
\begin{equation}\label{eq:d0}
d\pi\circ dL=d\pi\,.
\end{equation}

Let $g=\hol(\gamma)$ and $\bar\tau_l$ be the $D^l$-parallel transport
along $\gammab^{-1}$ we have the following commutative diagram
\fb{
\begin{equation}\label{eq:diag}
\begin{CD}
T_{p_0}M                @=     T_{p_0}M                 @>\hol_l(\gamma)>>      T_{p_0}M\\
@Ad\pi AA                       @Ad\pi AA                                      @Ad\pi AA\\
T_{\tilde p_0}\tilde M   @>dL>> T_{L(\tilde p_0)}\tilde M  @>\bar\tau_l>>         T_{\tilde p_0}\tilde M\\
@Vd(dev)VV                        @Vd(dev)VV                                  @Vd(dev)VV\\
T_{x_0}AdS_3            @>dg>> T_{g(x_0)}AdS_3           @>\tau_l(g(x_0))^{-1}>> T_{x_0}AdS_3
\end{CD}~.
\end{equation}
}

Indeed, the commutativity of the squares on the upper row is easy to check.
The commutativity of the second square of the second lower row relies on the fact
that $dev$ sends $D^l$-parallel vector field on $\tilde M$
to $D^l$-parallel vector field on $AdS_3$. Finally the commutativity of the first square
of the lower row follows from the fact that, by definition of holonomy, $g\circ dev=dev\circ L$.

By diagram (\ref{eq:diag}),
identifying $T_{p_0}M$ with $T_{x_0}AdS_3$ through the map
$d\pi(\tilde p_0)\circ (d(dev)(\tilde p_0))^{-1}:
T_{x_0}AdS_3\rightarrow T_{p_0}M$, we have that
\fb{
$\hol_l(\gamma)=\tau_l(g(x_0))^{-1}\circ dg(x_0)=g_l$ and analogously
$\hol_r(\gamma)=\tau(g(x_0))^{-1}\circ dg(x_0)=g_r$}.
\end{proof}

\subsection{The left and right metrics} \label{ssc:lrmetrics}

Every smooth curve $V(t)=(x(t),v(t))$ in $TM$ can be regarded as a
vector field along the curve $x(t)=\pi(V(t))$, so we can consider its covariant derivative
with respect to $D^l$ and $D^r$.

There is a splitting of $T(TM)$ associated with the connections $D^l$ and $D^r$.
Namely we have
\[
   T(TM)=T^V(TM)\oplus H^l=T^V(TM)\oplus H^r
\]
where
\begin{itemize}
\item $T^V(TM)$ is the vertical tangent space, that is the tangent space to the fiber (it
is independent of the connection),
\item $H^l$ and $H^r$ are the horizontal spaces (depending on the connection):
a vector $\xi\in T_{(p_0,v_0)}(TM)$ lies in $H^l$ (resp. $H^r$) if and only if
there exists a $D^l$-parallel (resp. $D^r$-parallel) curve $V(t)=(x(t),v(t))$
with $\dot V(0)=\xi$.
\end{itemize}

Each of these splittings provides  a  linear projection
\[
P^l,P^r:T_{(p,v)}(TM)\rightarrow T^V_{(p,v)}(TM)=T_pM
\]
and we easily see that
$P^l(\xi)=\frac{D^lV}{dt}(0)$ whereas $P^r(\xi)=\frac{D^rV}{dt}(0)$ where
$V(t)=(x(t),v(t))$ is any curve in $TM$ such that $\dot V(0)=\xi$.

Notice that if $(x,v)\in T^{1,t}M$ (cf. Definition \ref{df:unittmlk}) and $\xi\in T_{(x,v)}(T^{1,t}M)$, we can construct
the curve $V(t)=(x(t),v(t))$ so that $\langle v(t),v(t)\rangle=-1$.
Since $D^l$ and $D^r$ are compatible with the metric, we get that
$P^l(\xi)$ and $P^r(\xi)$ are orthogonal to $v$ in $T_xM$, so either they are $0$
or they are space-like.

\begin{defi}
We call $M_l$ and $M_r$ the two degenerate metrics (everywhere of rank 2)
defined on $T^{1,t}M$ as follows:
\[
M_l(\xi)=||P^l(\xi)||^2\,,~~M_r(\xi)=||P^r(\xi)||^2\,.
\]
\end{defi}

By construction, $M_l$ and $M_r$ are symmetric quadratic forms on the
tangent space of $T^{1,t}M$, and they are semi-positive, of rank $2$ at
every point.

We will derive a more concrete expressions of metrics $M_l$ and $M_r$
that will be useful in the sequel.  We use a natural identification
based on the Levi-Civita connection $\nabla$ of $M$:
$$ \forall (x,v)\in T^{1,t}M\,,\quad T_{(x,v)}(T^{1,t}M)\simeq T_xM\times
v^\perp\subset T_xM\times T_xM~. $$ In this identification, given
$v'\in v^\perp$, the vector $(0,v')$, considered as a vector in
$T(T^{1,t}M)$, corresponds to a ``vertical'' vector, fixing $x$ and
moving $v$ according to $v'$.  And, given $x'\in T_xM$, the vector
$(x',0)$, considered as a vector in $T(T^{1,t}M)$, corresponds to a
``horizontal'' vector, moving $x$ according to $x'$ while doing a
parallel transport of $v$ (for the connection $\nabla$).

Notice that $P^l(0,v')=P^r(0,v')=v'$: indeed by definition
there exists a curve $V(t)=(x,v(t))$ in $T_{x}M$ whose derivative in $0$ is
$(0,v')$, and we easily see that its covariant derivative (for any connection),
coincides with $v'$.

On the other hand, given a vector $(x',0)$ in $T_{(x,v)}TM$,
it can be extended to a curve $V(t)=(x(t),v(t))$ which is $\nabla$-parallel
and such that $\dot x(0)=x'$.
So we have
\[
\frac{D^l}{dt}V(t)=\frac{DV}{dt}+v(t)\times \dot x(t)\,,
\]
\fb{and} we conclude that $P^l(x',0)=v\times x'$. Analogously $P^r(x',0)=-v\times x'$.
So we conclude that
$$ M_l((x',v'),(x',v')) = \| v' + v\times x'\|^2~,~~
    M_r((x',v'),(x',v')) = \| v' - v\times x'\|^2~. $$

\begin{lemma}
With those definitions:
\begin{itemize}
\item $M_l$ and $M_r$ vanish on the integral curves of the geodesic flow of $M$.
\item $M_l$ and $M_r$ are invariant under the
geodesic flow of $M$.
\end{itemize}
\end{lemma}

\begin{proof}
We denote by $\phi_\bullet:TM\rightarrow TM$ the geodesic flow on $TM$.
Let us notice that the geodesic equation of the connection $\nabla$
coincides with the geodesic equation of $D^l$ and $D^r$. Thus
$\phi_\bullet$ can be regarded as the geodesic flow of the connection
$D^l$ ($D^r$) as well.

Since the orbits of the geodesic flow are tangent to the horizontal
space $H^l$ (resp. $H^r$),
$M_l$ and $M_r$ vanish on the direction tangent to the geodesic flow.

Given a point $(x,v)\in TM$ and $\xi\in T(TM)$, let us consider any curve $V(t)=(x(t),v(t))$ such that
$\xi=\dot  V(0)$. Putting $W(s,t)=\phi_s(V(t))=(y(s,t), w(s,t))$, by definition we have that
$y(\bullet,t)$ is a geodesic for any fixed $t$ and that $\frac{\partial y}{\partial s}(s,t)=w (s,t)$.
In particular $\frac{D^lW}{ds}=0$.

By definition,
\[
\frac{d\,}{ds}M_l(d\phi_s(\xi))=\frac{d}{ds}\left\langle\frac{DW}{dt}, \frac{DW}{dt}\right\rangle(s,0)\,.
\]
On the other hand, since $D^l$ is flat we have
\[
\frac{d\,}{ds}\left\langle\frac{DW}{dt},
\frac{DW}{dt}\right\rangle=2\left\langle\frac{D}{dt}\frac{DW}{ds},\frac{DW}{dt}\right\rangle=0
\]
and this shows that $M_l(d\phi_s(\xi))$ is constant.
\end{proof}

\begin{defi}
We denote by $G(M)$ the space of time-like maximal geodesics in $M$,
and by $m_l$ and $m_r$ the degenerate metrics on $G(M)$ induced by $M_l$ and
$M_r$, respectively.
\end{defi}

\begin{lemma}\label{local:lm}
$(G(M),m_l\oplus m_r)$ is locally isometric to $\HH^2\times \HH^2$.
\end{lemma}

\begin{proof}
Since the statement is local, we may suppose that $M$ is simply connected.
Let us fix a point $x_0\in M$.  Notice that the set of  time-like unit-vector at $x_0$,
say $T^{1,t}_{x_0}M$, is a space-like surface in $T_{x_0}M$ which is isometric
to the hyperbolic plane. We isometrically identify $T^{1,t}_{x_0}M$ with $\mathbb H^2$.

Let us consider the maps $\phi^l,\phi^r:T^{1,t}M\rightarrow T^{1,t}_{x_0}M=\mathbb H^2$ defined
by using the parallel transport for $D^l$ and $D^r$.

We claim that $M_l=(\phi^l)^*(g_\mathbb H)$ and $M_r=(\phi^r)^*(g_{\mathbb H})$.

To check the claim, let $V(t)=(x(t),v(t))$ for $t\in(-\epsilon,\epsilon)$ be
 any curve in $TM$ with $V(0)=(x,v)$ and $\dot V(0)=\xi$.
There is a homotopy $\sigma(t,s):(-\epsilon,\epsilon)\times[0,1]\rightarrow M$
such that $ \sigma(t,0)=x(t)$ and $\sigma(t,1)=x_0$.
The field $V$ can be uniquely extended to a field $W$ on $\sigma$ so that
$\frac{D^lW}{ds}=0$.

By definition we have that $W(t,1)=\phi^l(V(t))$, so $d\phi^l(\xi)=\frac{D^lW}{dt}(0,1)$.
In particular we have
\[
\left\langle d\phi^l(\xi), d\phi^l(\xi)\right\rangle=\left\langle\frac{D^lW}{dt}(0,1),\frac{D^lW}{dt}(0,1)\right\rangle\,.
\]
On the other hand, since $D^l$ is flat we have
\[
\frac{d}{ds}\left\langle\frac{D^lW}{dt}, \frac{D^lW}{dt}\right\rangle=
2\left\langle\frac{D^l\,}{dt}\frac{D^lW}{ds},\frac{D^lW}{dt}\right\rangle=0\,,
\]
so we deduce that
\[
\left\langle d\phi(\xi), d\phi(\xi)\right\rangle=\left\langle\frac{D^lW}{dt}(0,0), \frac{D^lW}{dt}(0,0)\right\rangle=
\left\langle\frac{D^lV}{dt}(0), \frac{D^lV}{dt}(0)\right\rangle=||P^l(\xi)||^2\,.
\]

We can obtain in the same way that the pull-back of the hyperbolic metric through $\phi^r$ is
$M^r$. In particular, considering the map $\phi(v)=(\phi^l(v),\phi^r(v))\in\mathbb H^2\times\mathbb H^2$,
$M_l\oplus M_r$ is the pull-back of the sum of hyperbolic metrics through $\phi$.

Notice that the orbits of the geodesic flow are horizontal for both $D^l$ and $D^r$ (this because
geodesics of $\nabla$ coincide with geodesics of $D^l$ and $D^r$), so we deduce that
$\phi$ is constant on the orbits of geodesic flow, so it induces a map
\[
\bar\phi: G(M)\rightarrow\mathbb H^2\times\mathbb H^2
\]
and we have that $\bar\phi^*(g_{\mathbb H}\oplus g_{\mathbb H})=m_l\oplus m_r$.

In order to conclude we should  prove that $\bar\phi$ is a local diffeomorphism.
This is equivalent to showing that $m_l\oplus m_r$ is non-degenerate.
Since the tangent space of $G(M)$ is the quotient of the tangent space of $T^{1,t}M$ along
the line tangent to the orbit of the geodesic flow, it is sufficient to prove that vectors
$\xi\in T_{(x,v)}(T^{1,t}M)$ such that $M_l(\xi)=M_r(\xi)=0$ are tangent to the orbit of the geodesic flow.

Taking such a  $\xi=(x',v')$, we deduce that $v'+v\times x'=v'-v\times x'=0$, so $v'=0$
and $v\times x'=0$. Thus $\xi=(x',0)$ with $x'$ parallel to $v$, and this is the condition to be tangent
to the geodesic flow.
\end{proof}

\begin{remark}
The proof of Lemma \ref{local:lm} shows that in the general case the
developing map of $m_l\oplus m_r$ is the map
\[
\bar\phi:\tilde G(M)=G(\tilde M)\rightarrow\mathbb H^2\times\mathbb H^2
\]
described above.
Its holonomy is given by the pairs of representations $(\hol_l,\hol_r)$
(up to the identification of $\pi_1(G(M))$ with $\pi_1(M)$).
\end{remark}

Note that there is another possible way to obtain the same hyperbolic metrics
$m_l$ and $m_r$, using the identification of $\HH^2\times \HH^2$
with $PSL(2,\R)\times PSL(2,\R)/O(2)\times O(2)$. We do not elaborate on this point
here since it appears more convenient to use local considerations.

\subsection{Transverse vector fields and associated hyperbolic metrics}
\label{ssc:transverse}

The construction of the left and right hyperbolic
surfaces associated to an AdS
3-manifold is based on the use of a special class of
surfaces, endowed with a unit time-like vector field behaving
well enough, in particular with respect to the singularities.


Let us first consider the case without particle:

\begin{defi} \label{df:muleft}
Let $M$ be any smooth AdS manifold.
Let $S\subset M$ be a space-like surface, and let $V$ be a field of time-like
unit vectors defined along $S$. It is {\bf transverse} if
\fb{for all $x\in S$, the maps $v\mapsto D^l_vV$ and $v\mapsto D^r_vV$ have rank
$2$}.
\end{defi}

It is not essential to suppose that $S$ is space-like, and the weaker
topological assumption that $S$ is isotopic in $M$ to a space-like
surface would be sufficient. The definition is restricted to space-like
surface for simplicity.

\begin{defi} \label{df:610}
We still assume that $M$ is a regular AdS spacetime. Let $S\subset M$ be a 
space-like surface, and let $V$ be a transverse vector
field on $S$. Let $\delta:S\rightarrow G(M)$ be the map sending a
point $x\in S$ to the time-like geodesic parallel to $V$ at $x$.
We call $\mu_l:=\delta^*m_l$ and $\mu_r:=\delta^*m_r$.
\end{defi}

Notice that the field $V$ can be regarded as a map $S\rightarrow T^{1,t}M$,
and we have $\mu_l=V^*(M_l)$ and $\mu_r=V^*(M_r)$. In particular we easily see
that
\[
  \mu_l(v,v)=||D^l_vV||^2 ~,~~\mu_r(v,v)=||D^r_vV||^2~.
\]
So the metrics $\mu_l$ and $\mu_r$ are not degenerate.

If $\tau_l,\tau_r:T^{1,t}\tilde M\rightarrow T^{1,t}_{x_0}\tilde M=\mathbb H^2$
are the maps obtained by parallel transport for $D^l$ and $D^r$ on the universal covering,
as in Lemma \ref{local:lm}, we have that the developing map of $\mu_l$ is the map
$dev_l(x)=\tau_l(\tilde V(x))$ where $\tilde V$ is the lift of $V$ on $\tilde S\subset \tilde M$.
Analogously $dev_r(x)=\tau_r(\tilde V(x))$ is a developing map for $\mu_r$.

Now consider the case where $M$ contains some particles, and denote by
$M_{reg}$ the smooth part of $M$. Let $S$ be a space-like surface meeting the particles
orthogonally and let $V$ be a transverse vector field on $S_{reg}=M_{reg}\cap S$. The field $V$
defines two hyperbolic metrics $\mu_l$ and $\mu_r$ on $S_{reg}$ with holonomy $\hol_l$ and
$\hol_r$ respectively. However   in general the behavior
of the metrics around the particles can be very degenerate.
We say that $V$ is a transverse vector field on $S$ if it satisfies the following
conditions,  which ensure that $\mu_l$ and $\mu_r$ are hyperbolic metrics
with cone singularities  around the particles.

\begin{defi}\label{part:defi}
Let $T$ be a particle in $M$ and $p$ be the intersection point of $T$
with $S$.  We consider a neighborhood $W$ of $p$ that is obtained by
glueing a wedge $\hat W\subset AdS_3$ of angle $\theta$ as explained
in \cite[3.7.1]{colI}. The intersection $\Delta=S\cap W$ corresponds
to a surface $\hat \Delta$ on $\hat W$, with $p$ corresponding to a
point $\hat p$ on $\hat \Delta$.
The surface $S$ is smooth around $p$ if $\hat \Delta$ can be extended
to a smooth surface in a neighborhood of $\hat p$.

The vector field $V$ is transverse around a particle $T$ if the
following conditions are satisfied:
\begin{itemize}
\item $V$ extends to a unit vector field in $p$ tangent to $T$.
\item
The induced vector field $\hat V$ on $\hat W$ extends to a smooth
vector field in a neighborhood of $\hat W$ in $AdS_3$,
\item The rank of $D^l\hat V$ and $D^r\hat V$ at $\hat p$ is $2$.
\end{itemize}
\end{defi}

First let us exhibit a large class of vector fields that
satisfy the conditions of this definition\fb{.}

\begin{lemma}\label{norm:lm}
If the cone angle around the particle is $\theta\in (0,2\pi), \theta\neq\pi$ and
$S$ is a smooth surface orthogonal to the particle,
its unit normal vector field satisfies the conditions of Definition \ref{part:defi} at $p$.
\end{lemma}

\begin{proof}
The first two conditions are easily verified.
Let us check the third condition.

Using the fact that $\hat V$ is the pull-back of a vector field on $W$,
at $\hat p$ we have that
\[
   \nabla_{R(v)} \hat V=R\nabla_v\hat V~,
\]
where $R:T_{\hat p}AdS_3\rightarrow T_{\hat p}AdS_3$ is the rotation of angle
$\theta$ with axis the line tangent to $T$, and $v$ is a vector orthogonal to $T$ and tangent
to the boundary of $W$.
If $\theta\neq \pi$
this implies that $\nabla \hat V(\hat p)=\lambda I+\mu J$ where $J$ is the rotation of $\pi/2$ around the line
tangent to $T$.

On the other hand, if $V$ is the normal field of $S$, we have that $\nabla \hat V$ is a self-adjoint
operator on $TS$. Thus we have that the skew-symmetric part of $\nabla\hat V$ must vanishes
at $p$. It follows that $\nabla \hat V(\hat p)=\lambda I$.

In particular, since the transformation $v\mapsto \hat V(p)\times v$
coincides with $J$, we deduce that
\[
D^l\hat V(\hat p)=\lambda I+J~,\qquad\qquad~D^r\hat V(\hat p)=\lambda I-J~,
\]
and the third condition in the definition follows.
\end{proof}

\begin{remark}
\label{rk:conepi}
When the cone angle is $\pi$, the same conclusion follows provided that $S$ is convex around
the particle.
\end{remark}

\begin{prop} \label{pr:mulr}
Let $M$ be a space-time with particles and $S$ be a closed smooth
surface orthogonal to the particles. If $V$ is a transverse field on
$S$, then $\mu_l$ and $\mu_r$ are hyperbolic metrics with cone
singularity. Moreover if $p$ is the intersection point of $S$ with a
particle of angle $\theta$, then $p$ is a cone point for both $\mu_l$
and $\mu_r$ of the same angle.
\end{prop}

\begin{proof}
The metrics $\mu_l$ and $\mu_r$ are defined on the smooth part
$S_{reg}=S\cap M_{reg}$ and are hyperbolic by Lemma \ref{local:lm}.

Since $V$ is smooth at any particle, it is easy to check that $D^lV$
and $D^rV$ are uniformly bounded operators of $TS_{reg}$. This implies
that $\mu_\bullet$ is bi-Lipschitz to the first fundamental form. In
particular the completion of $(S_{reg},\mu_\bullet)$ is canonically
identified with $S$.

Let $p$ be the intersection point of $S$ with a particle.  A
neighborhood $W$ of $p$ in $M$ is obtained by glueing the boundary of
a wedge $\hat W$ of angle $\theta$ in $AdS_3$.  Let $\Delta=S\cap W$
and $\hat\Delta$ the corresponding surface in $\hat W$.

By hypothesis, $\hat\Delta$ is a sector of a smooth surface $\Sigma$
around $\hat p$ in $AdS_3$ orthogonal to the edge of $\hat W$, and
$\hat V$ can be extended to a smooth vector field on $\Sigma$.

In particular the metrics $\hat\mu_\bullet$ on $\hat\Delta$ extend to
smooth hyperbolic metrics on $\Sigma$ and $(\Delta,\mu_\bullet)$ is
obtained by gluing the boundary of $(\hat\Delta,\hat\mu_\bullet)$ by a
rotation around $\hat p$.

Let us consider in $T_{\hat p}AdS_3$ the sector $P$ of vectors tangent
to curves contained in $\hat\Delta$. It is clearly a sector of angle
$\theta$ for the AdS metric. If we show that $P$ is a sector of
angle $\theta$ also for $\mu_\bullet$, then the result will easily
follows.

Notice that if $\theta=\pi$, then $P$ is a half-plane, so the angle is
$\pi$ for any metric.  If $\theta\neq\pi$, as in Lemma \ref{norm:lm},
we have that $D_\bullet V$ is a conformal transformation at $\hat p$. 
Since $\mu_l(\bullet,\bullet)=\langle D^l_\bullet V,
D^l_\bullet V\rangle$ we see that the angle of $P$ with respect to
$\mu_l$ is still $\theta$ (and analogously for $\mu_r$).
\end{proof}

\subsubsection*{Note.}

The reason this paper is limited to manifolds with massive particles
--- rather than more generally with interacting singularities as in \cite{colI} --- is that we
do not at the moment have good analogs of those surfaces with transverse
vector fields when other singularities, e.g. tachyons, are present.

\subsection{A special case: good surfaces}
\label{ssc:good}

The previous construction admits a simple special case, when the
time-like vector field is orthogonal to the surface (which then
has to be space-like).

\begin{defi}
Let $M$ be an AdS manifold with interacting particles. Let $S$ be a
smooth space-like surface. $S$ is a {\it good} surface if:
\begin{itemize}
\item it does not contain any interaction point,
\item it is orthogonal to the particles,
\item the curvature of the induced metric is negative,
\item the intersections of $S$ with the particles of angle $\pi$ are locally convex.
\end{itemize}
\end{defi}

Note that, given a good surface $S$, one can consider the equidistant
surfaces $S_r$ at distance $r$ on both side. For $r$ small enough
(for instance, if $S$ has principal curvature at most $1$,
when $r\in (-\pi/4,\pi/4)$), $S_r$ is a smooth
surface, and it is also good. So from one good surface one gets a
foliation of a neighborhood by good surfaces.

The key property of good surfaces is that their unit normal vector
field is a transverse vector field, according to the definition given
above. This simplifies the picture since the left and right metrics
are defined only in terms of the surface, without reference to a
vector field. However the construction of a good surface seems to
be quite delicate in some cases, so that working with a more general
surface along with a transverse vector field is simpler.

\begin{lemma}
Let $S$ be a good surface, let $u$ be the unit normal vector field on
$S$, then $u$ is a transverse vector field.
\end{lemma}

\begin{proof}
Let $x\in S_{reg}$ and let $v\in T_xS$. By definition,
$$ D^l_v u = \nabla_vu + u\times v = -Bv + Jv~, $$
$$ D^r_v u = \nabla_vu - u\times v = -Bv - Jv~, $$
where $B$ is the shape operator of $S$ and $J$ is the complex structure
of the induced metric on $S$. If $S$ is a good surface then its
induced metric has curvature $K<0$. But
$\det(-B\pm J) = \det(B)+1 = -K$, so that $D^l_vu$ and $D^r_vu$
never vanish for $v\neq 0$. This means precisely that $u$ is a
transverse vector field along $S_{reg}$. 

Now the lemma follows from Lemma \ref{norm:lm} and Remark \ref{rk:conepi}.
\end{proof}

\subsubsection*{Example}

Let $s_0$ be a space-like segment in $AdS_3$ of length
$l>0$. Let $d_0, d_1$ be disjoint time-like lines containing the endpoints
of $s_0$ and orthogonal to $s_0$,
chosen so that the angle between the (time-like) plane $P_0$
containing $s_0, d_0$ and the (time-like) plane $P_1$ containing $s_0$
and $d_1$ is equal to some $\theta\in\mathbb R$.
Let $W_0$ (resp. $W_1$) be wedges with
axis $d_0$ (resp. $d_1$) not intersecting $s_0$ or $d_1$ (resp $d_0$)
(see Figure \ref{fig:1}).

Let $M_{\theta}$ be the space obtained from $AdS_3\setminus W_0\cup W_1$ by
gluing isometrically the two half-planes in the boundary of $W_0$
(resp. $W_1$), and let $M_{ex}:=M_\theta$ for $\theta=l$.
We will see that $M_{ex}$ does not contain any good surface, or
even any surface with a transverse vector field.

\begin{figure}
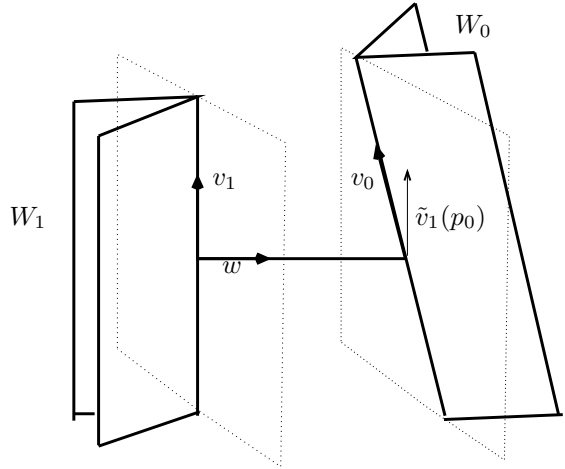

\begin{center}
\input example.pstex_t
\end{center}
\caption{A simple example with no transverse vector field.}
\label{fig:1}
\end{figure}

Let $x_0$ be the end-point of $s_0$ contained in $d_0$.  Let us
identify $SO_0(2,1)$ with the isometry group of $T_{x_0}AdS_3$.  By
definition the holonomy $g_0$ of a loop of $M$ around $d_0$ is a
rotation of axis $d_0$. So $(g_0)_l=(g_0)_r$ are elliptic
transformations with fixed point the vector $v_0$ tangent to $d_0$

On the other hand, if $g_1$ is the holonomy of a loop around $d_1$, 
$(g_1)_l$ and $(g_1)_r$ are elliptic
transformations with fixed point the vectors $\tau_l(v_1)$ and
$\tau_r(v_1)$ respectively, where $v_1$ is the direction tangent to
$d_1$ at the end-point $p_1\in d_1\cap s_0$.  Now if $\tilde v_1$ is
the $\nabla$-parallel field on $s_0$ extending $v_1$ on $s_0$ and $w$
is the unit vector field tangent to $s_0$ and pointing to $p_0$ we
have by our assumption
\[
        v_0=\ch(\theta)\tilde v_1(p_0)\pm \sh(\theta) w\times\tilde v_1(p_0)
\]
(where the sign $\pm$ depends on the way $d_1$ is turned with respect to $d_0$).

On the other hand, the $D^l$ and $D^r$-parallel extensions of $v_1$ along $s_0$
 are respectively
\[
\tau_l(v_1)(p)=\ch(s)\tilde v_1(p)+\sh(s)w\times \tilde v_1(p)\, ,\qquad
\tau_l(v_1)(p)=\ch(s)\tilde v_1(p)+\sh(s)w\times \tilde v_1(p)
\]
where $s$ is the distance on $s_0$ between $p$ and $p_1$.

In particular, assuming $\theta=l$, either $\tau_l(v_1)$ or
$\tau_r(v_1)$ coincides with $v_0$.  This implies that either the left
or the right holonomy is an elementary representation, so it cannot be
the holonomy of a hyperbolic disk with two cone singularities.

Notice that if $\theta<l$, $M_{\theta}$ does contain a
space-like surface with a transverse vector field (we leave the
construction to the interested reader) but with a left hyperbolic
metric, say $\mu_l(\theta)$, which has two cone singularities which
``collide'' as $\theta\rightarrow l$. (This can be seen easily by
taking a surface which contains $s_0$.) If $M_{ex}$ admitted a
surface with a transverse vector field, it could have only one cone
singularity (as it is seen by considering the limit $M_\theta\rightarrow
M_{ex}$), this is impossible.

Note that $M_{ex}$ is obviously not globally hyperbolic, and it contains no
closed space-like surface, it was chosen for its simplicity.

\subsection{Changing the transverse vector field and the space-like slice}
\label{sub:changing}

In this section we will consider the same framework as in the previous
section.  In particular we fix an AdS manifold $M$ with particles, a
closed space-like surface in $M$ and a transverse vector field $V$ on
$S$.  We will investigate how the metrics $\mu_\bullet$ change when
deforming the surface $S$ and the vector field.  The first basic
result is that the class of isotopy of $\mu_\bullet$ is independent of
$S$ and $V$. More precisely, if $S$ and $S'$ are two different
surfaces in $M$ we will prove that there \fb{exist} isometries
$\phi_\bullet:(S,\mu_\bullet)\rightarrow(S',\mu'_\bullet)$ such that
the induced map $\phi_*:\pi_1(S_{reg})\rightarrow\pi_1(S'_{reg})$ makes the
following diagram commutative (up to conjugation)
\[
\begin{CD}
\pi_1(S_{reg})@>i_*>>\pi_1(M_{reg})\\
@V\phi_*VV       @VIdVV\\
\pi_1(S'_{reg})@>i'_*>>\pi_1(M_{reg})
\end{CD}~.
\]
Notice that the commutativity of the diagram determines the isotopy
class of $\phi$.  More geometrically $\phi$ is isotopic to any map
$S\rightarrow S'$ obtained by following the flow of any time-like
field tangent to the particles.

\begin{lemma} \label{lm:surface}
Given $S$, $\mu_l$ and $\mu_r$ do not depend (up to isotopy) on the choice
of the transverse vector field $V$. Moreover, $\mu_l$ and $\mu_r$ do not
change (again up to isotopy) if $S$ is replaced by another surfaces
isotopic to it.
\end{lemma}

The proof uses a basic statement on hyperbolic surfaces with
cone singularities. Although this result might be well known,
we provide a  proof for completeness.

\begin{lemma} \label{lm:holonomy}
A closed hyperbolic surface with cone singularities of angle less than
$2\pi$ is uniquely determined by its holonomy.
\end{lemma}

Let $S$ be a closed surface with marked points $x_1,\cdots, x_n$,
let $\theta_1, \cdots, \theta_n\in (0,2\pi)$, and let $\mu_0, \mu_1$ be two
hyperbolic metrics on $S$ with cone singularities of angles $\theta_i$
(resp. $\theta'_i$) on the $x_i$, $1\leq i\leq n$. We suppose that
$\mu_0$ and $\mu_1$ have the same holonomy, and will prove that $\mu_0$
is isotopic to $\mu_1$.

The holonomy of a short curve around a singular point $x_i$ is a rotation
of angle the cone angle at $x_i$. Since $\mu_0$ and $\mu_1$ have the same
holonomy at $x_i$, we see already that $\theta_i=\theta'_i$ for all
$1\leq i\leq n$.

\begin{sublemma}\label{ssl:triang}
There exists a triangulation $T_0$ of $S$ with vertices equal to the
$x_i$, with as edges segments that are geodesic for $\mu_0$.
\end{sublemma}

\begin{proof}
A basic point on metrics with cone singularity with angle less than
$2\pi$ is that any two cone points are connected by a minimizing
geodesic which does not pass through any other cone singularity.
Given a curve $\gamma_0$ between two singular points $x_i$ and
$x_j$, there is a unique curve $\gamma_1$ between $x_i$ and $x_j$
which can be deformed to $\gamma_0$ in the complement of the $x_k$,
and which has minimal length among all such curves. However
$\gamma_1$ can go through some of the $x_k$.

It follows that there is a graph $\Gamma$ embedded in $S$ with vertices
$x_i$ having edges segments with are geodesic for $\mu_0$.

Cutting $S$ along $\Gamma$, we get a surface $\hat S$ with piecewise
geodesic boundary such that the total angle at any vertex of the boundary is
in $(0,2\pi)$. Any loop $c$ in $\hat S$  centered at some vertex of $\partial\hat S$
can be deformed to a geodesic curve (that may possibly contain some segment
of the boundary). In particular, cutting $\hat S$ along geodesic curves which are not
isotopic to the boundary, we eventually get a decomposition of $S$ in
hyperbolic disks $D_1,\ldots, D_n$ which have piecewise geodesic boundary, such that
vertices of $\partial D_i$ correspond to cone points of $S$.

To conclude the proof it is sufficient to show that each $D_i$ admits a geodesic
triangulation with vertices corresponding to the vertices of its boundary.
We use an inductive argument on the number of vertices of $\partial D_i$.

Take any vertex $p$ of $D_i$ and consider the set  $\mathcal A$
of end-points of maximal geodesic segments embedded
in $D_i$ starting at $p$. By maximality, $\mathcal A$ is a subset of
$\partial D_i$. We distinguish two cases. First we suppose that $\mathcal A$ contains
a vertex $q$ of $\partial D_i$ which is not adjacent to $p$. In this case, cutting $D_i$
along the maximal segment joining $p$ to $q$, we decompose $D_i$ into $2$ disks 
with less vertices. In the other case, any segment starting from $p$ and contained in
the interior of $D_i$ must meet a single edge of $D_i$. This means that there exists a geodesic
triangle in $D_i$ with a vertex at $p$ whose edges contains the edges of $\partial D_i$
adjacent to $p$. In this case there is a segment in the interior of $D_i$ joining the
vertices adjacent to $p$. Thus we can cut $D_i$ into a triangle $T'$ contained in $T$
and another disk $D'$ with less vertices.
\end{proof}

Note that the argument given in the previous proof for the existence of a triangulation
could be replaced by another argument, based on Voronoi diagrams, which is somewhat
simpler in the setting considered here. The reason why we favored the slightly
more involved argument used here is that we will repeat the same argument below
in the slightly different setting of surfaces with (convex) boundary. The type
of argument used here works directly for surfaces with boundary, while the
argument based on Voronoi diagrams is less directly applicable there. 

\begin{proof}[Proof of Lemma \ref{lm:holonomy}]

We define a {\it 3-disk} in $S$ as a disk in $S$ containing no marked
point in its interior and exactly three marked points in its
boundary. Those disks are considered up to homotopy of $S$ fixing the
marked points $x_i$. Let $D$ be such a disk, containing in its
boundary the marked points $x_i,x_j,x_k$. Considering the restriction
to $D$ of the developing map of the regular part of $(S,\mu_0)$ we associate to $x_i,x_j$
and $x_k$ a triple $(x'_i,x'_j,x'_k)$ of points in $\HH^2$, defined up
to global isometry of $\HH^2$, as well as a disk $D'$ containing
$x'_i,x'_j$ and $x'_k$ in its boundary ($D'$ is defined up to homotopy
in $\HH^2$ fixing $x'_i, x'_j$ and $x'_k$).

Notice that $D'$ and $(x'_i, x'_j,x'_k)$
are uniquely determined by the holonomy of $\mu_0$ only, because the
holonomy determines the 
cone angles at $x'_i$, $x'_j$, $x'_k$ (through the holonomies 
of loops around these points) and the 
distance between $x'_i$ and $x'_j$ (through
the trace of the holonomy of the boundary of a small neighborhood
of the segment of $\partial D$ between $x_i$ and $x_j$) and
similarly for $x'_j$ and $x'_k$ and for $x'_k$ and $x'_i$.

We say that the 4-tuple $(D',x'_i, x'_j, x'_k)$ is {\it realizable} if
$D'$ can be deformed to a triangle (with geodesic boundary) with
vertices $x'_i,x'_j$ and $x'_k$, without displacing $x'_i, x'_j$ and
$x'_k$. Clearly if $(x_i, x_j,x_k)$ are the vertices of a triangle $D$
of a geodesic triangulation $T$ of $(S,\mu_0)$, then
$(D',x'_i,x'_j,x'_k)$ is realizable. But conversely, if, for any face
$D$ of a triangulation $T$ with vertices $x_i,x_j$ and $x_k$, $(D',
x'_i,x'_j,x'_k)$ is realizable, then considering the developing map
of the metric shows that $T$ can be realized as a geodesic
triangulation.

Since the condition for a 3-disk to be realizable depends only on the
holonomy, the geodesic triangulation $T_0$ of $(S,\mu_0)$ constructed
in Sublemma \ref{ssl:triang}
also corresponds to a geodesic triangulation of $(S,\mu_1)$.
Moreover the length of the edges is the same for the two metrics,
because we have seen that the length of an edge is determined by
the holonomy. So $\mu_1$ is isotopic to $\mu_0$.
\end{proof}

\begin{proof}[Proof of Lemma \ref{lm:surface}]
For the first point consider another transverse  vector field $V'$
on $S$, and let $\mu'_l,\mu'_r$ be the hyperbolic metrics defined
on $S$ by the choice of $V'$ as a transverse vector field.
Let $\gamma$ be a closed curve on the complement of the singular
points in $S$. The holonomy of $\mu'_l$ (resp. $\mu'_r$) on
$\gamma$ is equal to the holonomy
of $D^l$ (resp. $D^r$) acting on the hyperbolic plane, identified
with the space of oriented time-like unit vectors at a point of
$S$. So $\mu_l$ and $\mu'_l$ (resp. $\mu_r$ and $\mu'_r$) have the
same holonomy, so that they are isotopic by Lemma \ref{lm:holonomy}.

The same argument can be used to prove the second part of the lemma.
Let $\gamma_1$ be a closed curve on $S_1$ which does not intersect the
singular set of $M$, and let $\gamma_2$ be a closed curve on $S_2$
which is isotopic to $\gamma_1$ in the regular set of $M$. The holonomy
of $M$ on $\gamma_1$, $\hol(\gamma_1)$, is equal to the holonomy of
$M$ on $\gamma_2$, $\hol(\gamma_2)$. But $\hol=(\hol_l,\hol_r)$ by
Lemma \ref{lm:rho}, and $\hol_l,\hol_r$ are the holonomy representations
of the left and right hyperbolic metrics on $S_1$ and on $S_2$ by
Lemma \ref{lm:rho}. Therefore, $(S_1,\mu_l)$ has the same holonomy
of $(S_2,\mu_l)$, and $(S_1,\mu_r)$ has the same holonomy as
$(S_2,\mu_r)$. The result therefore follows by Lemma
\ref{lm:holonomy}.
\end{proof}

Note that a weaker version of this proposition is proved
as \cite[Lemma 5.16]{minsurf} by a different argument.
The notations $\mu_l,\mu_r$ used here are the same as in
\cite{cone}, while the same metrics
appeared in \cite{minsurf} under the notations $I^*_\pm$. Those
metrics already appeared, although implicitly only, in Mess' paper
\cite{mess}. 
As we have mentioned in \fb{Section} \ref{sub:GH-AdS}, 
\jms{this} paper considers globally hyperbolic AdS manifolds,
which are the quotient of a maximal convex subset $\Omega$ of $AdS_3$ by
a surface group $\Gamma$ acting by isometries on $\Omega$. The identification
of \fb{$SO_0(2,2)$} with $PSL(2,\R)\times PSL(2,\R)$ then determines
two representations of $\Gamma$ in $PSL(2,\R)$ with maximal Euler
number, so that they define hyperbolic metrics. It is proved in
\cite{minsurf} that those two hyperbolic metrics correspond precisely
to the left and right metrics considered here.

\begin{remark}
In general the isometries $\phi_l:(S,\mu_l)\rightarrow(S',\mu'_l)$ and
$\phi_r:(S,\mu_r)\rightarrow (S',\mu'_r)$ are different.
This implies that the pair $(\mu_l,\mu_r)$ is not uniquely determined
up to isotopy (acting on both the factors).
\end{remark}


In the remaining part of this section we will show that
any transverse unit vector field $V$ on $S$ can be extended to a unit vector
field on a neighborhood  $\Omega$ of $S$ such that
\begin{itemize}
\item it is tangent to the particle,
\item it is transverse on any  space-like surface $S'$ contained in $\Omega$,
\item the map $\phi:S\rightarrow S'$ obtained by following the orbits of $V$
is an isometry for both $\mu_l$ and $\mu_r$.
\end{itemize}

In fact there exists $\epsilon>0$ such that the map
\[
  F:S\times(-\epsilon,\epsilon)\ni (p,t)\mapsto \exp_p(tV(p))\in AdS_3\,.
\]
is well-defined and it is a diffeomorphism onto some neighborhood
$\Omega$ of $S$ in $AdS_3$.
Notice that if $p$ is the intersection point of $S$ with some particle,
$F(p,t)$ lies on the particle for every $t$. Moreover, by the assumption on
$V$ it is easy to check that $F$ is a diffeomorphism around the particles.

Clearly the map $F$ induces on $\Omega$ a foliation by time-like geodesics
parallel to $V$, so we can consider the induced map
\[
\hat\delta:\Omega\rightarrow G(M)
\]
and the bilinear forms $\hat\mu_l=\hat\delta^*(m_l)$,
$\hat\mu_r=\hat\delta^*(m_r)$.

Since $\hat \delta(F(p,t))=\delta(p)$, where $\delta:S\rightarrow
G(M)$ is the map defined in Definition \ref{df:610}, we have that
$F^*(\hat\mu_\bullet)=\pi_S^*(\mu_\bullet)$ where
$\pi_S:S\times(-\epsilon,\epsilon)\rightarrow S$ is the projection.

In particular, $\hat\mu_\bullet$ is non-degenerate on every plane that
is not tangent to $\hat V$.  If $S'$ is any space-like surface, $\hat
V$ is transverse to it and we have that the induced metrics
$\mu'_\bullet=\hat\mu_\bullet|_{S'}$. Finally, the map $S'\rightarrow
S$ sending $q\in S'$ to the intersection point of the geodesic leaf
through $q$ with $S$ turns out to be an isometry for both $\mu'_l$,
$\mu_l$ and $\mu'_r$ and $\mu_r$.


\subsection{Left and right metrics on the future of a collision point}

We consider now the case where $S$ is a space-like surface with a transverse
vector field  in a AdS  spacetime  which contains
a unique collision point $p$. Without loss of generality we suppose
 $p$ in the past of $S$.
Clearly $I^+(p)\cap S$ is a disk $D$ with $k$ singular  points where $k$ is the number
of particles starting from $p$.

\begin{defi}
A connected open  subset $U$ of a hyperbolic surface $S$ is
convex if any path  $c$ contained in $U$ can be deformed
to a geodesic path in $U$ keeping the endpoints of $c$ fixed.
\end{defi}

The goal of this section is to prove the following proposition

\begin{prop}\label{pr:isomdsk}
There are convex disks $D_l, D_r$ isotopic to
$D$ in $S$ such that $(D_l, \mu_l)$ is isometric to $(D_r,\mu_r)$.
\end{prop}
(Let us stress that here the isotopies are supposed not to displace
the cone points.)

\fb{
First we show that the statement is true for the holonomies.

\begin{lemma}\label{lm:isomdsk}
The holonomies
of $\mu_l$ and $\mu_r$ restricted to $\pi_1(D_{reg})$
are conjugated.
\end{lemma}
\begin{proof}
We use the fact that the holonomies of $\mu_l$ and $\mu_r$
are the left and right factors of the holonomy of the AdS structure,
as stated in Lemma \ref{lm:rho}.

So we have to show that the restriction of the holonomy of 
$h$ to $\pi_1(D_{reg})$ is conjugated to a diagonal representation
into $SO_0(2,2)=SO_0(2,1)\times SO_0(2,1)$.

If $\Sigma$ is the link of the collision point, the inclusion
$\pi_1(D_{\reg})\rightarrow \pi_1(M_{\reg})$ can be factored as the composition
$\pi_1(D_{reg})\to\pi_1(\Sigma_{reg})\to\pi_1(M_{reg})$. So it is sufficient to prove
that the restriction o $h$ to $\pi_1(\Sigma_{reg})$ is conjugated to a  diagonal representation.

On the other hand, this clearly follows since the holonomy $h$
restricted to $\pi_1(\Sigma_{reg})$ fixes a point (see Lemma \ref{lm:stab}).
\end{proof}

Notice that Lemma \ref{lm:isomdsk} is not sufficient to conclude
the proof of Proposition \ref{pr:isomdsk}, since we have to point out 
concrete disks $D_l$ and $D_r$ such that $(D_l,\mu_l)$ and $(D_r,\mu_l)$ are
isometric. To that aim we will use the same triangulation argument as in Lemma \ref{lm:holonomy}.
}

  The main difference is that in
this case $D_l$ and $D_r$ have boundary, so we need to select them
carefully: the key point in the proof of Lemma \ref{lm:holonomy} is
that pairs of points are joint by a minimizing geodesic. So in order
to apply the same argument we need the convexity of $D_l$ and $D_r$.

\begin{proof}
Take a sequence $D_n$ of disks isotopic to $D$ such that the $\mu_l$-length of
$\partial D_n$ converges to the infimum
of the  $\mu_l$-lengths of boundary curves of disks isotopic to $D$.

$\partial D_n$ converges to a $\mu_l$-geodesic graph $\Gamma$ with
vertices at cone points of $D$.  In fact, locally around each point
$x$ of $\Gamma$ we see two regions $\Omega_1$ and $\Omega_2$ in
$S\setminus\Gamma$. By the minimizing property we easily see that if
$\Omega_i$ is in the limit of $S\setminus D_n$, then the angle
contained in $\Omega_i$ with vertex at $x$ is bigger than $\pi$. Since
singularities are supposed to have angles in $(0,2\pi)$ we easily see
that one of the following possibilities occurs:
\begin{itemize}
\item $D$ contains only one cone point, and $\Gamma$ coincides with it.
\item $D$ contains exactly two cone points and $\Gamma$ is a segment
  with vertices at cone points.
\item $D$ contains more than $2$ cone points and $\Gamma$ is a circle
  bounding a convex disk $D'_l$ such that the regular neighborhoods of
  $D'_l$ are isotopic to $D$.
\end{itemize}

\fb{In the first case it is sufficient to define $D_l$, $D_r$ to be 
disks of radius $\epsilon$ around the cone point for $\mu_l$ and $\mu_r$ 
respectively.

In the second case, notice that the $\mu_l$-length of $\Gamma$, say $a$,
 is determined by the holonomy of $\mu_l$ on $D_{reg}$.
Since the holonomies of $\mu_l$ and $\mu_r$ on $D_{reg}$ coincide
$\Gamma$ can be deformed to an arc $\Gamma'$ which is $\mu_r$ geodesic and such that
the $\mu_r$-length of $\Gamma'$ is also $a$. Then a $\mu_l$-regular neighborhood of
$\Gamma$ and a $\mu_r$-regular neighborhood of $\Gamma'$ are isometric.
}

In the third case, we can construct a $\mu_l$-geodesic triangulation of $D'_l$ as
in Proposition \ref{lm:holonomy}. The shape of this triangulation just depend on
the holonomy of the disk, and this implies that each triangle of this
triangulation can be deformed to a $\mu_r$-triangle.

This implies that there exists an isometric embedding
$(D'_l,\mu_l)\rightarrow (S,\mu_r)$ which is isotopic to the inclusion
$D \rightarrow S$.  Thickening a bit $D'_l$ we get a convex disk $D_l$
isotopic to $D$ such that $(D'_l,\mu_l)$ admits an isometric embedding
(isotopic to the identity) into $(S, \mu_r)$.
\end{proof}

\begin{remark}
In general the intersection of two convex disks is not connected.
On the other hand, the proof of Proposition \ref{pr:isomdsk} shows
that any convex disk isotopic to $D$ must contain $D'_l$. This implies
that if $D_1, D_2$ are convex disks isotopic to $D$, then  $D_1\cap D_2$
contains a convex disk isotopic to $D$.
\end{remark}

\subsubsection*{Example}

If only two particles, $p_1$ and $p_2$, collide, the corresponding
cone points are at the same distance in the left and right hyperbolic
metric of $\Omega$; more precisely, there are two segments of the same
length, one in the left and one in the right hyperbolic metric of
$\Omega$, joining the cone points corresponding to $p_1$ and to $p_2$.
Moreover the length of those segments is equal to the ``angle''
between $p_1$ and $p_2$ at $c$, i.e., to the distance between the
corresponding points in the link of $c$.

\section{Surgeries at collisions} \label{sc:3}
\label{sc:surgeries}

We now wish to understand how the left and right hyperbolic metrics
change when a collision occurs.

\subsection{Good spacial slices}

The first step in understanding AdS manifolds with colliding particles
is to define more easily understandable pieces.

\begin{defi}
Let $M$ be an AdS manifold with colliding particles. A {\bf spacial slice}
in $M$ is a subset $\Omega$ such that
\begin{itemize}
\item there exists a closed surfaces $S$
with marked points $x_1,\cdots, x_n$ and a homeomorphism
$\phi:S\times [0,1]\rightarrow \Omega$,
\item $\phi$ sends $\{ x_1,\cdots,x_n\}\times [0,1]$ to the singular set of
$\Omega$,
\item $\phi(S\times \{ 0\})$ and $\phi(S\times \{ 1\})$ are space-like
surfaces.
\end{itemize}
$\Omega$ is a {\bf good} spacial slice if in addition
\begin{itemize}
\item it contains a space-like surface with a transverse vector field.
\end{itemize}
\end{defi}

Hence, we have constructed at the end of Section \ref{sub:changing} a good spacial slice in the neighborhood
of any Cauchy surface equipped with a transverse vector field. Observe that spacial slices do not contain interactions.

It is useful to note that Lemma \ref{lm:surface}, along with its
proof, applies also to surfaces with boundary, with a transverse
vector field, embedded in a good spacial slice. Such surfaces
determine the holonomy of the restriction of the left and right
metrics to surfaces with boundary, as explained in the following
remark.  The proof is a direct consequence of the arguments used in
Section \ref{sc:leftright}, and more specifically in the proof of Lemma
\ref{lm:surface}.

\begin{remark} \label{rk:surface}
Let $\Omega$ be a good spacial slice, let $D\subset \Omega$ be a space-like
surface with boundary, and let $u'$ be a transverse vector field on $D$.
Then $u'$ determines a left and a right hyperbolic metric, $\mu'_l, \mu'_r$
on $D$, as for closed surfaces above. Moreover for any closed curve
$\gamma$ contained in $D$, the holonomies of $\mu'_l$ and $\mu'_r$
on $\gamma$ are equal respectively to the left and right parts of the
holonomy of $\gamma$ in $M$.
\end{remark}

\subsection{Surgeries on the left and right metrics} \label{ssc:surgeries}

In this section we consider in details how the left and right metrics
change when a collision occurs. The first step is to define some simple
notions of surgery on hyperbolic surfaces with cone singularities, and on
pairs of such surfaces. We will later prove that those surgeries are exactly
those that can happen on the left and right metrics of spacial slices of
an AdS manifold with particles when a collision occurs.

\subsubsection{Surgery on hyperbolic surfaces}

The basic building block of the surgeries considered here is a simple
operation where one replaces a disk, in a hyperbolic surface with cone
singularities, by another disk with only one singularity.

\begin{defi}
Let $S_-$ and $S_0$ be two hyperbolic cone-surfaces, and
let $D_-\subset S_-$ be homeomorphic to an open disk. We say that $S_0$
is obtained from $S_-$ by {\bf collapsing} $D_-$ if
\begin{itemize}
\item there exists an isometric embedding $i:S_-\setminus D_-\rightarrow S_0$,
\item $S_0 \setminus i(S_-\setminus D_- )$ is homeomorphic to an open disk and contains
exactly one cone singularity $s_0$, of angle $\theta\in (0,2\pi)$.
\end{itemize}
We call $s_0$ the collapsed singularity of $S_0$.
\end{defi}

Note that the geometry of the disk $S_0 \setminus i(S_-\setminus D_- )$ depends only
on the geometry of $D_-$. In other terms, there is a
hyperbolic disk $D_{0,-}$ with exactly one cone singularity, depending only on
$D_-$, and an isometric embedding $j_-:D_{0,-}\rightarrow S_0$ such that
$S_0\setminus j(D_{0,-})=S_-\setminus D_-$.

We now introduce the surgery on {\it pairs} of hyperbolic
cone-surfaces, which corresponds --- as it will be seen below --- to what
occurs to the left and right hyperbolic metrics of an AdS manifold
with particles when a particle collision occurs. The basic idea is
that a disk surgery is done on both $S_-^l$ and $S_-^r$, collapsing
the {\it same} disk $D_-$ (up to isometry of course) and yielding the
same disk $D_+$ (again up to isometry). However an additional
condition is necessary, stating that the ``relative position'' of
$D_-$ and $D_+$ is the same on the left and on the right side.

We consider a surface $S$ with a couple of hyperbolic cone metrics
$\mu_l, \mu_r$ and we suppose that (up to isotopy) they coincide in a
neighborhood of a singular convex disk $D$. Moreover we will
suppose that the holonomy of $\partial D$ (for both $\mu_l$ and
$\mu_r$) is elliptic of angle $\theta\in (0,2\pi)$ and that a collar neighborhood of $\partial
D$ admits an isometric embedding (for both $\mu_l$ and $\mu_r$)
$i:\partial D\rightarrow H_\theta$, where $H_\theta$ is the model of
the singularity of angle $\theta$.  The image of those embeddings bound
a disk $D_0$ in $H_\theta$ containing the singular point.  We consider
now the surface $S_0$ obtained by cutting from $S$ the disk $D$ and
pasting the disk $D_0$ using $i$ as glueing map. Notice that the
metrics $\mu_l$ and $\mu_r$ glue to the metric $\mu$ of $D_0$,
yielding two singular metrics which coincide on a disk $D$.

\begin{defi}\label{df:dcollaps}
We say that the triple $(S_0, \mu_l^0, \mu_r^0)$ is obtained by $(S,
\mu_l, \mu_r)$ by collapsing $D_-$.
\end{defi}

Changing $\mu_l, \mu_r$ by two different isotopies (but requiring that
they coincide on some disk $D'$ isotopic to $D$), we get another
collapsed surface $S_0', {\mu'_l}^0, {\mu'_r}^0$.  Clearly there are
natural isometries
\[
  \phi_l:(S_0,\mu^0_l)\rightarrow (S_0', {\mu'_l}^0)\qquad
  \phi_r:(S_0, \mu^0_r)\rightarrow (S_0', {\mu'_r}^0)~.
\]

Though those isometries are in general different, the following lemma
establishes that when $D$ contains at least two cone points, $\phi_l$
and $\phi_r$ coincide in a neighborhood of the collapsed singularity.

\begin{lemma}\label{lm:rig}
If $D$ contains at least two singular points, then $\phi_l$ and $\phi_r$
coincide in a neighborhood of the collapsed singularity of $S_0$.
\end{lemma}

Notice that if $D$ contains only a singular point, then the statement
is false. In fact, the proof of Lemma \ref{lm:rig} is based on
the following simple fact which clearly depends on the fact that $D$ contains
at least two cone points.

\begin{sublemma}\label{slm:rig}
If $D$ is a convex disk in a hyperbolic surface $S$ with at least two cone
singularities and if $\sigma: D\rightarrow S$ is an isometric immersion
isotopic to the identity, then $\sigma$ is the identity.
\end{sublemma}

\begin{proof}[Proof of Sublemma \ref{slm:rig}]
Let $p_1$, $p_2$ cone points of $D$.
The map $\sigma$ fixes the two cone points, so it fixes any geodesic arc
which joins $p_1$ to $p_2$ and this easily implies that it fixes
every point.
\end{proof}

\begin{proof}[Proof of Lemma \ref{lm:rig}]
Assuming $D'\subset D$ and
$\mu_l=\mu'_l$, $\mu_r=\mu'_r$,
the statement is clearly true.

Similarly, the statement is true also
assuming there exists a diffeomorphism $u$ of $S$ isotopic to
the identity such that $\mu'_l=u^*(\mu_l)$, $\mu'_r=u^*(\mu_l)$ and
$D'=u^{-1}(D)$.

So it is sufficient to consider the case where
$\mu'_l=\mu_l$, $\mu'_r=u^*(\mu_r)$ for some diffeomorphism $u$
isotopic to the identity.

Notice that in this case the disk $u(D')$ is $\mu_r$-convex.
In particular there is a $\mu_r$-convex disk $\Delta$ contained in
$D\cap u(D')$ and isotopic to $D$. Notice that $\mu_r=\mu_l$  on $\Delta$,
so this disk is also $\mu_l$-convex.
The restriction of $u^{-1}$ to $\Delta$ is an isometric embedding of
$(\Delta,\mu_l=\mu_r)\rightarrow(D',\mu_l=\mu'_r)\subset(S,\mu_l)$.
By Sublemma \ref{slm:rig},$u|_{\Delta}=Id$.

Now, let $(\hat S_0, \hat \mu_l, \hat \mu_r)$ be the surface obtained by
collapsing $\Delta$ on $(S,\mu_l, \mu_r)$ and let $(\hat S'_0,\hat\mu_l, \hat\mu_r)$
be the surface obtained by collapsing $\Delta$ on $(S, \mu'_l, \mu'_r)$.
Notice that the isometries $\hat\phi_l, \hat\phi_r:\hat S'_0\rightarrow\hat S_0$
extend respectively the identity and $u$ on $S''=S\setminus\Delta$. So they coincide on
$\partial \Delta$ and this shows that they coincide on the disk
$\Delta_0''=\hat S_0\setminus S''$.

On the other hand, the isometries
$\psi_l, \psi_r:S_0\rightarrow \hat S_0$ coincide in a neighborhood
of the collapsed point (because $\Delta\subset D$) and similarly do
the isometries $\psi'_l,\psi'_r:S_0'\rightarrow\hat S_0'$.

The statement follows
since $\phi_l=  (\psi'_l)^{-1}\circ\hat\phi_l\circ\psi_l$
and $\phi_r=  (\psi'_r)^{-1}\circ\hat\phi_r\circ\psi_r$.
\end{proof}

\begin{defi}  \label{df:double}
Let $S_-$ and $S_+$ be two surfaces, let $\mu_l^-,\mu_r^-$ be hyperbolic
cone metrics on $S_-$  sharing the same singular locus $\sigma_-$
and let $\mu_l^+, \mu_r^+$ be hyperbolic cone metrics
on $S_+$ with singular locus $\sigma_+$.
We say that $(S_+,\mu_l^+, \mu_r^+)$ is obtained from
$(S_-, \mu_l^-, \mu_r^-)$ by a {\bf double surgery} if
up to changing those metrics in their isotopy classes,
the following conditions are satisfied:
\begin{enumerate}
\item There are embedded singular disks $D_-\subset S_-$ and $D_+\subset S_+$ such that
$\mu_l^+$ and $\mu_r^+$ coincide in a neighborhood of $D_+$ and
$\mu_l^-$ and $\mu_r^-$ coincide in a neighborhood of $D_-$.
\item The corresponding collapsed surfaces $(S_0, \mu^0_l, \mu^0_r)$ and
$(\hat S_0,\hat\mu^0_l, \hat\mu^0_r)$ are isometric, that is there are two
isometries
\[
   \phi_l:(S_0, \mu^0_l)\rightarrow (\hat S_0, \hat \mu_l^0)\qquad
  \phi_r:(S_0, \mu^0_r)\rightarrow (\hat S_0, \hat\mu_r^0)~.
\]
\item
$\phi_l$ and $\phi_r$ are homotopic and
coincide in a neighborhood of the collapsed points.
\end{enumerate}
The disks $D_-$ and $D_+$ are called the surgery disks, whereas the
homotopy class of $\phi_l$ and $\phi_r$ is called the identification map.
\end{defi}

\begin{remark}
Condition (3) in the above definition needs some explanation.
Notice that since $\mu^0_l$ coincides with $\mu^0_r$ in a neighborhood of the collapsed point
of $S_0$, and $\hat \mu^0_r$ coincides wih $\hat\mu^0_r$ in a neighborhood of the collapsed
point of $\hat S_0$ we have that in general $\phi_r^{-1}\phi_l$ is an isometry in a neighborhood
of the collapsed point, that is, there is a number $\theta$ such that $\phi_r^{-1}\phi_l$ is a 
rotation of angle $\theta$. We require that $\theta=0$.  In the simple case where only
one surgery is sufficient, this condition means that the \emph{same} surgery transforms
$\mu_l^-$ into $\mu_l^+$ and $\mu_r^-$ into $\mu_r^+$.

The following example shows a case where condition (3) is not satisfied (see Figure \ref{fg:1}).
Take a surface $(S,\mu)$ with two cone points such that the holonomy around the cone points is 
elliptic and there is a constant curvature circle $c$ which bounds a disk $D$ containing the cone points.
 Now take $S_-=S$, $\mu_l^-, \mu_r^-=\mu$, $S_+=S$ and $\mu_l^+=\mu$. Finally define
 $\mu_r^+$ by twisting the disk $D$ of angle $\theta_0$ (this is possible since $\partial D$
 has constant curvature). More formally, the metric $\mu_r^+$ is constructed as follows:
let $\tau:S\setminus\mathring D\rightarrow S\setminus \mathring D$ be the identity
outside a collar of $\partial D$ and a rotation of angle $\theta$ on $\partial D$. Then,
$\mu_r^+=\tau^*(\mu)$ on $S\setminus D$ and $\mu$ on $D$.

In this case, if $S_0$ is the surface obtained from $S_-$ 
by replacing $D$ by a disk
$D_{\theta_0}$ with only one cone point of angle $\theta_0$,  
and $\hat S_0$ is the surface obtained by replacing  $D$ by 
$D_{\theta_0}$, then
$\phi_l$ is the identity map, whereas $\phi_r$ is $\tau$ 
outside $S\setminus D$ and is a rotation
of angle $\theta$ on $D$.

Note however that if $D_+$ (or $D_-$) contains only one cone singularity, 
then condition (3) is always satisfied. 
\end{remark}

Heuristically, the fact that $(S_+,\mu^+_l, \mu^+_r)$ and $(S_-, \mu^-_l, \mu^-_r)$ are related by
a double surgery means that there is a surface $S_0, \mu_l, \mu_r$ and two singular disks $D_-$ and $D_+$ such that $S_+$ is obtained by replacing a disk in $S_0$ containing a cone point
by $D_+$ and $S_-$ is obtained by replacing another disk containing the same cone point by
$D_-$.
It is tempting to simplify this definition, and to replace directly $D_-$ by $D_+$
without going through the intermediate step of $S_0$ with only one singularity instead of
either $D_-$ or $D_+$. It appears however that it is not always possible
to do this direct surgery --- an example is described in Appendix \ref{ap:A} of a
situation where a double surgery as defined here cannot be replaced by a simple
surgery where one topological disk is replaced by another.
Theorem \ref{tm:main} shows that the relevant notion when considering 
collisions of particles in AdS spacetimes is that of double surgery,
rather than the simpler notion of simple surgery on both the left and
right metrics.

\begin{figure}
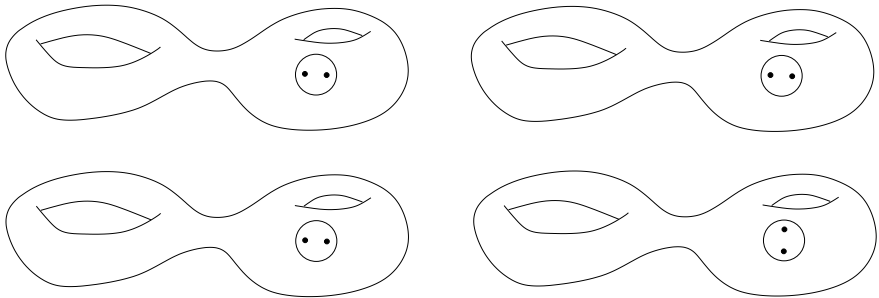

\begin{center}
\input dsurgery.pstex_t
\caption{An example of two pairs of surfaces which satisfy all
  conditions but (3) in the definition of double surgery.}
\label{fg:1}
\end{center}
\end{figure}


\subsubsection{Setting and main statement}

We now consider a more precise setting.
Let $\Omega$ be an AdS manifold with interacting particles,
containing exactly one collision point $p$, which we suppose has positive
mass. Suppose that $p$ is the future
endpoint of $n$ particles $s_1, \cdots, s_n$ and the past
endpoint of $m$ particles $s'_1, \cdots, s'_m$. Let $\theta_1,\cdots,
\theta_n$ be the cone singularities at the $s_i$, and let
$\theta'_1,\cdots, \theta'_m$ be the cone singularities at the $s'_j$.

Suppose that $\Omega$ is the union of two good space-like slices
$\Omega_-$ and $\Omega_+$, such that:
\begin{itemize}
\item they have disjoint interior,
\item the future boundary of $\Omega_-$ is equal to the past boundary
of $\Omega_+$,
\item $\Omega_-$ contains $s_1, \cdots, s_n$ and $\Omega_+$ contains $s'_1, \cdots, s'_m$.
\end{itemize}
We call $S_-$ a space-like surface in $\Omega_-$ with a transverse vector field $u_-$,
and $S_+$ a space-like surface in $\Omega_+$ with a transverse vector field
$u_+$. Let $\mu^\pm_l,\mu^\pm_r$ be the left and right hyperbolic metrics defined on
$S_\pm$ by $u'_\pm$.

\begin{prop} \label{pr:collision}
Under those conditions, the triple $(S_+, \mu^+_l, \mu^+_r)$ is obtained from the triple 
$(S_-, \mu^-_l,\mu^-_r)$ by a double surgery.

The surgery disks are in the isotopy class of $D_+=I^+(p)\cap S_+$ and $D_-=I^-(p)\cap S_-$, respectively,
whereas the identification maps are in the isotopy class of the map $S_+\setminus D_+\rightarrow S_-\setminus D_-$ 
obtained by following any timelike flow  sending $\partial D_+$ to $\partial D_-$.
\end{prop}

\fb{
\emph{Note:}
In the proof of this proposition and in the rest of the paper
we will consider developing maps and holonomies
of a singular manifold $X$. 
So we need to consider the universal covering and the fundamental
groups of the regular part of $X$, whereas in general we will not be interested in
the fundamental group and the universal covering of $X$.
For this reason, from now on
$\pi_1(X)$ and $\tilde X$  will denote respectively
the fundamental group and the universal covering of the regular part of $X$.
}

\begin{proof}
By Proposition \ref{pr:isomdsk},
up to changing $\mu^+_l$ and $\mu^+_r$ by an isotopy,
we may suppose that they
coincide around a disk $D_+$ containing the singular points $p_i=s_i\cap S_+$.
Analogously we may suppose that $\mu^-_l$ and $\mu^-_r$ coincide on  a disk
$D_-$ which contains the singular points $p'_i=s'_i\cap S_-$.

By the positivity of the mass of the collision point, the holonomy
of $\partial D_+$ for $\mu^+_\bullet$ is elliptic of angle $\theta_0\in (0,2\pi)$. 
In particular there
is an embedding of a neighborhood of $\partial D_+$ into the model space  $H_{\theta_0}$
of the cone angle $\theta_0$.

So we can consider the surface $(S_0, \mu_l^0, \mu_r^0)$ obtained by collapsing $D_+$
on $(S_+, \mu_l^+, \mu_r^+)$.  Analogously let $(\hat S_0, \hat\mu_l^0, \hat\mu_r^0)$
be the surface obtained by collapsing $D_-$ on $(S_-, \mu_l^-, \mu_r^-)$.

Let us regard the fundamental group as the set of covering transformations 
on the universal cover.
In particular, any lifting on the universal covering of the inclusions
\[
  (S_-\setminus D_-)\rightarrow S_-\rightarrow M\qquad
  (S_+\setminus D_+)\rightarrow S_+\rightarrow M
\]
determines inclusions $\pi_1(S_-\setminus D_-)\rightarrow\pi_1(S_-)\rightarrow\pi_1(M)$
and $\pi_1(S_+\setminus D_+)\rightarrow\pi_1(S_+)\rightarrow\pi_1(M)$.

Since $S_-\setminus D_-$ and $S_+\setminus D_+$ are isotopic in $M$, we may fix
those liftings
\[
 \widetilde{(S_-\setminus D_-)}\rightarrow \tilde S_-\rightarrow \tilde M\qquad
  \widetilde{(S_+\setminus D_+)}\rightarrow \tilde S_+\rightarrow \tilde M
\]
so that $\pi_1(S_-\setminus D_-)$ is identified to $\pi_1(S_+\setminus D_+)$ as subgroups
of $\pi_1(M)$.

Finally since the inclusions $S_-\setminus D_-\rightarrow \hat S_0$ and
$S_+\setminus D_+\rightarrow S_0$ are homotopy equivalence, we may fix  identifications
between $\pi_1(S_-\setminus D_-)$ and $\pi_1(\hat S_0)$ and $\pi_1(S_+\setminus D_+)$
and $\pi_1(S_0)$. Notice that those identifications are unique up to conjugation and the choice
of concrete ones is equivalent to choosing liftings $\widetilde{S_+\setminus D_+}\rightarrow\tilde S_0$
and $\widetilde{S_-\setminus D_-}\rightarrow\tilde{\hat S}_0$ of the natural inclusions.

Notice that through these identifications, the holonomies of $\mu_l^0, \mu_r^0$ coincide
with the holonomies of $\hat\mu_l^0, \hat\mu_r^0$, so, by Proposition \ref{lm:holonomy}, there exist
isometries
\[
\phi_l:(\hat S_0, \hat\mu^0_l)\rightarrow (S_0, \mu^0_l)\qquad
\phi_r:(\hat S_0, \hat\mu^0_r)\rightarrow (S_0, \mu^0_r)
\]
which admit liftings to the universal covering
$\tilde\phi_l, \tilde\phi_r:\tilde {\hat{S_0}}\rightarrow\tilde S_0$
which act trivially on the fundamental groups
\[
   \tilde\phi_l\circ\gamma\circ(\tilde\phi_l)^{-1}=\tilde\phi_r\circ\gamma\circ(\tilde\phi_r)^{-1}=\gamma ~,
\]
where we are using the identification $\pi_1(S_0)=\pi_1(\hat S_0)$ fixed above.

Notice that $\phi_l$ and $\phi_r$ are isotopic, since they induce the same map
on the fundamental groups.

In order to prove condition (3) in Definition \ref{df:double},
we also fix liftings $\tilde D_+\rightarrow \tilde S_+$ and $\tilde D_-\rightarrow \tilde S_-$,
so that $\tilde D_+\cap \widetilde{S_+\setminus D_+}$ and
$\tilde D_-\cap\widetilde{S_-\setminus D_-}$ are the images of liftings $c_+$ and $c_-$
of $\partial D_+$ and $\partial D_-$, respectively. We may moreover suppose that the stabilizers
of $c_+$ and $c_-$ in $\pi_1(M)$ are the same $\mathbb Z$-subgroup generated by $\gamma_0$.

We fix developing maps
\[
  dev_+^\bullet:(\tilde S_+, \mu_+^\bullet)\rightarrow\mathbb H^2
\]
so that they coincides on $\tilde D_+$.
Notice that the restriction of $dev_+^\bullet$ on $\widetilde{S_+\setminus D_+}$
extends to a developing map
$dev_0^\bullet :\tilde S_0\rightarrow\mathbb H^2$.

Clearly, $\widehat{dev}_0^\bullet:=dev_0^\bullet\circ\phi_\bullet$
is a developing map for $\hat\mu^0_\bullet$.

Finally let $dev_-^l, dev_-^r$ be the developing maps
of $\mu_-^\bullet$ which coincide with $\widehat{dev}_0^\bullet$ on
$\widetilde{S_-\setminus D_-}$.
Notice that on $c_-$ the holonomy representations of the maps
$dev_-^l$ and $dev_-^r$ differ exactly by the rotation $\tilde R\in PSL(2, \R)$
which corresponds to $\phi_r\phi_l^{-1}$ around the cone point.
In particular, if $\gamma\in\pi_1(D_-)$ we have
\[
\hol_r(\gamma)=\tilde R\hol_l(\gamma)\tilde R^{-1}~.
\]

Notice that
$\pi_1(\Omega)$ is the amalgamated product of $\pi_1(S_+)$ and $\pi_1(S_-)$
with the identification of $\pi_1(S_+\setminus D_+)$ and $\pi_1(S_-\setminus D_-)$
described above.

The holonomies of the developing maps $\dev_\pm^\bullet$ glue to a pair of representations
\[
 (\hol_l, \hol_r):\pi_1(\Omega)\rightarrow PSL_2(\mathbb R)^2
\]
which coincide with the holonomy representation of $\Omega$.

If $\Sigma$ is the link around the collision point, then 
$\pi_1(\Sigma)$ is the amalgamated product of $\pi_1(D_-)$ and $\pi_1(D_+)$.
If $\gamma\in\pi_1(D_+)$ then $\hol_r(\gamma)=\hol_l(\gamma)$, whereas
if $\gamma\in\pi_1(D_-)$ then $\hol_r(\gamma)=\tilde R\hol_l(\gamma)\tilde R^{-1}$.

Imposing that the restrictions of the representations $\hol_l$ and $\hol_r$ on $\pi_1(\Sigma)$
are the same, we obtain that $\tilde R=Id$.
\end{proof}

\begin{figure}
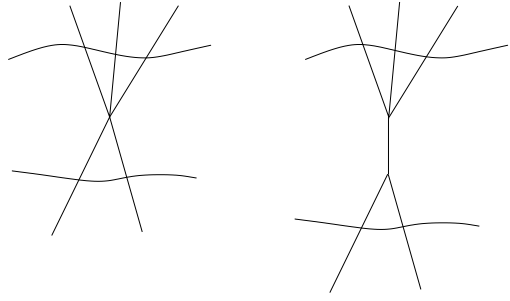

\input transex.pstex_t
\caption{Two different interaction possibilities}\label{fg:trans}
\end{figure}

\begin{remark}
Let us consider the examples in Figure \ref{fg:trans}. In the example on the left,
Proposition \ref{pr:collision} indicates that
the left and right metrics on the surface
below and the left and right metrics on the surface above are related by a double surgery

In the example on the right, this is no longer true. On the other hand, conditions (1) and (2) in the
definition of double surgery are still valid.
The same argument used in the proof of Proposition \ref{pr:collision} shows that in
this case the map $\phi_l\circ\phi_r^{-1}$ is a rotation
about the collapsed point of angle equal to the distance between the collision points.
\end{remark}

\subsection{Transverse vector fields after a collision}

It might be interesting to remark that the description made in
Proposition \ref{pr:collision} of the surgery on the left and right hyperbolic metrics corresponding
to a collision only holds -- and actually only makes sense -- if there is a
space-like surface with a transverse vector field both before and after the
collision. However the existence of such a surface before the collision does
not ensure the existence of one after the collision, even for simple
collisions.


A simple example of such a phenomenon can be obtained by an extension of
the example given in Section \ref{ssc:good} 
of an AdS space with two particles containing no
space-like surface with a transverse vector field. Consider the space
$M_\theta$ described in that example, with $\theta<l$, so that $M_\theta$
contains a space-like surface with a transverse vector field. This space
has two cone singularities, $d_0$ and $d_1$, each containing one of the
endpoints of $s_0$. It is now possible to perform on this space a simple
surgery as described in \cite[Section 7.1]{colI}
replacing the part of $d_1$ in
the past of its intersection with $s_0$ by two cone singularities,
say $d_2$ and $d_3$,
intersecting at the endpoint of $s_0$. This can be done in such a way
that the angle between the plane containing $s_0$ and $d_2$ and the
plane containing $s_0$ and $d_0$, is equal to $l$. The argument given
above for $M_{ex}$ then shows that there is no space-like surface
with a transverse vector field in a spacial slice before the collision.


\subsection{The graph of interactions}
\label{ssc:graph}

The previous section contains a description of the kind of surgery on
the left and right hyperbolic metrics corresponding to a collision of
particles. Here a more global description is sought, and we will associate
to an AdS manifold with colliding particles a graph describing the relation
between the different spacial slices. In all this part we fix an AdS manifold
with colliding particles, $M$.

Let $\Omega, \Omega'$ be two spacial slices in $M$. They are {\bf equivalent}
if each space-like surface in $\Omega$ is isotopic to a space-like surface
in $\Omega'$. Note that this clearly defines an equivalence relation on
the spacial slices in $M$.

\begin{defi} \label{df:good}
$M$ is a {\bf good} AdS manifold with colliding particles if any spacial
slice in $M$ is equivalent to a good spacial slice.
\end{defi}

Clearly if two good spacial slices are equivalent then their holonomies
are the same, so that their left and right hyperbolic metrics are isotopic
by Proposition \ref{lm:holonomy}.

Some of the examples
constructed by \cite[Proposition 7.7]{colI} are indeed good AdS manifolds
with colliding particles. (To obtain one such example, one can construct
an AdS manifold with colliding particles by a surgery on a Fuchsian space
with one particle, replacing a neighborhood of the particle by a tube
where two particles collide to become two new particles, so that the two 
particles are almost parallel both before and after the collision.)

\begin{defi}
Let $\Omega_-$ and $\Omega_+$ be two spacial slices in $M$.
They are {\bf adjacent} if the union of the compact connected
components of the complement of the interior of
$\Omega_-\cup \Omega_+$ in $M$ contains exactly one collision.
We will say that $\Omega_-$ is {\bf anterior} to $\Omega_+$ if
this collision is in the future of $\Omega_-$ and in the
past of $\Omega_+$.
\end{defi}

Note that this relation is compatible with the equivalence
relation on the spacial slices: if $\Omega_-$ is adjacent to
$\Omega_+$ and $\Omega'_-$ (resp. $\Omega'_+$) is equivalent
to $\Omega_-$ (resp. $\Omega_+$) then $\Omega'_-$ is adjacent
to $\Omega'_+$. Moreover if $\Omega_-$ is anterior to $\Omega_+$
then $\Omega'_-$ is anterior to $\Omega'_+$.

\begin{defi}
The {\bf graph of spacial slices} is the oriented graph associated
to a good AdS manifold with colliding particles $M$ in the following way.
\begin{itemize}
\item The vertices of $G$ correspond to the equivalence classes of
spacial slices in $M$.
\item Given two vertices $v_1, v_2$ of $G$, there is an edge between
$v_1$ and $v_2$ if the corresponding spacial slices are adjacent.
\item This edge is oriented from $v_1$ to $v_2$ if the spacial
slice corresponding to $v_1$ is anterior to the spacial slice corresponding
to $v_2$.
\end{itemize}
\end{defi}

\begin{remark}
Notice that two different admissible AdS structures in $\struct$
may have different graph of spacial slices. On the other hand, it is clear
that the graphs of spacial slices of spacetimes in a small neighborhood
of some fixed space $M\in\struct$ naturally contain the
graph of spacial slices of $M$.
\end{remark}

Note that, for the constructions that follow and in particular for Theorem \ref{tm:main},
it would be sufficient to require only, rather than Definition \ref{df:good}, that there is a path
on the graph of spatial slices of $M$ whose vertices are equivalent to good slices.

\subsection{The topological and geometric structure added to the graph of interactions}
\label{ssc:structure}

Clearly the graph of spacial slice is not in general a tree -- there might
be several sequences of collisions leading from one spacial slice to
another one. A simple example is given in Figure \ref{fg:graph}, where the
graph of a manifold with colliding particles is shown together with a
schematic picture of the collisions.

\begin{figure}[ht]
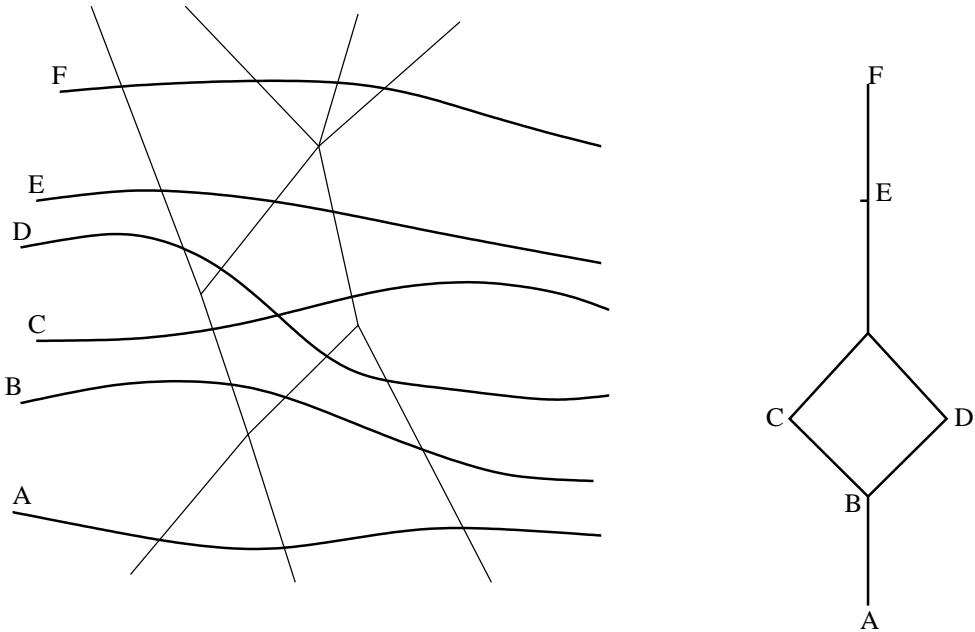

\begin{center}
\input graph.pstex_t
\end{center}
\caption{The graph of spacial slices.}
\label{fg:graph}
\end{figure}

The graph of spacial slices is clearly not sufficient to recover an AdS
manifold with colliding particles, additional data are needed.

\begin{defi}
A {\bf topological data} associated to an oriented graph is the choice of:
\begin{itemize}
\item For each vertex $v$, of a closed surface $S_v$ with $n$
marked points $p_1, \cdots, p_n$, and a $n$-tuple $\theta_v=
(\theta_1,\cdots, \theta_n)\in (0,2\pi)^n$.
\item For each oriented edge $e$ with vertices $e_-$ and $e_+$, of:
  \begin{enumerate}
  \item a homotopy class of disks $D_{e,+}\subset S_{e_+}$, where the
homotopies are in the complement of the marked points $p_i$ (or equivalently
they are homotopies of the complements of the $p_i$ in $S_v$),
  \item a homotopy class of disks $D_{e,-}\subset S_{e_-}$, where again
the homotopies fix the marked points,
  \item an isotopy class $i_e$ of homeomorphisms from
$S_{e_-}\setminus D_{e,-}$ to $S_{e_+}\setminus D_{e,+}$ sending
the marked points to the marked points.
  \end{enumerate}
\end{itemize}
\end{defi}

\begin{defi} \label{df:geometric}
A {\bf geometric data} associated to an oriented graph endowed with
a topological data is the choice, for each vertex $v$, of two
hyperbolic metrics $\mu_l(v),\mu_r(v)$ on $S_v$, with a cone singularity
of angle $\theta_i$ at $p_i$, so that, for each edge $e$ with endpoints $e_-$
and $e_+$, $(S_{e_+}, \mu_l(e_+),\mu_r(e_+))$ is obtained from
$(S_{e_-}, \mu_l(e_-),\mu_r(e_-))$ by
a double surgery with surgery disks $D_{e_+}$ and $D_{e_-}$ respectively,
 as seen in Definition \ref{df:double}.
\end{defi}

Given a good AdS space with colliding particles $M$ we can consider its graph
of collisions $\Gamma$, there is a natural topological and geometric data associated
to $M$ on $\Gamma$. Given a vertex $v$ of the graph of collisions $\Gamma$, it
corresponds to a good spacial slice $\Omega_v$ in $M$, and we take as $S_v$ a
space-like 
surface in $\Omega_v$. The marked points correspond to the intersections
of $S_v$ with the particles in $S_v$. By definition of a good spacial slice,
$S_v$ admits a transverse vector field, so one can define the left and right
hyperbolic metrics $\mu_l(v)$ and $\mu_r(v)$ on $S_v$ through Definition
\ref{df:610}. The fact that those two metrics are well defined follows from
Lemma \ref{lm:surface}.

Now consider an edge of $\Gamma$, that, is a collision between particles.
Let $e_-$ corresponds to the good spacial slice $\Omega_-$ in the past of the collision,
and $e_+$ to the good spacial slice $\Omega_+$ in the future of the collision.
It follows from Proposition \ref{pr:collision} that
$(S_{e_+}, \mu_l(e_+),\mu_r(e_+))$ is obtained from
$(S_{e_-}, \mu_l(e_-), \mu_r(e_-))$ by a double surgery.

\section{From the geometric data to the structure}
\label{sc:final}

In this section we fix a maximal good AdS  spacetime $M_0\in\struct$ with collision
and we consider the corresponding topological data $X$.
Let $\mathcal D(X)$ the set of geometric data with topological data $X$.
\fb{ An element of $\mathcal D(X)$ is basically a collection of
pairs of singular hyperbolic metrics $\mu_l(v)$ and $\mu_r(v)$  on $S_v$
for every vertex of the graph of interaction, where the cone singularities are fixed by the topological
data. Thus we can regard $\mathcal D(X)$ as a subset of the Cartesian product 
$\Pi_{v}\mathcal T(S_v, \theta_v)^2$, where $\mathcal T(S_v, \theta_v)$ denotes the Teichm\"uller
space of singular hyperbolic metrics with cone angles $\theta_1\ldots\theta_n$.
We will consider on $\mathcal D(X)$ the induced topology.
}

Notice that there is a neighborhood $\mathcal U$ of $M_0\in\struct$
such that:
\begin{itemize}
\item The graph of spacial slice of spacetimes in $\mathcal U$ contains the graph of
spacial slices of $M_0$.
\item The topological data of $M_0$ coincides with the restriction of the topological data
of any $M\in\mathcal U$ to the graph of $M_0$.
\item Every vertex of the graph of $M_0$ corresponds to a good spacial slice for
any structure of $\mathcal U$ (this because the transversality condition is open).
\end{itemize}
This defines a map
\[
    GD:\mathcal U\rightarrow\mathcal D(X)
\]
sending any structure $M\in\mathcal U$ to the corresponding geometric data.

\begin{theorem} \label{tm:main}
Up to shrinking $\mathcal U$, the map $GD$ is injective and open.
\end{theorem}

\begin{proof}
We will construct an open and injective map
\[
  H: \mathcal D(X)\rightarrow\mathcal R(g, T, \theta)
\]
such that the holonomy map $hol:\struct\rightarrow \mathcal R(g, T, \theta)$
factors as $hol=H\circ GD$. The conclusion of the proof will follow from the existence
of this map and from Theorem \ref{struct:thm}.

In order to construct the map $H$ we give a description of the fundamental group of $M$ by means of the topological data.
We fix a path in the graph of $M_0$, say $v_1,\ldots, v_n$,
 joining the initial vertex to the final vertex, and consider the corresponding sequence
 of good space-like surfaces $S_1,\ldots, S_n$.

As in the proof of Proposition \ref{pr:collision} we may fix lifting of the natural inclusions
\[
   \widetilde{S_k\setminus D_{k,-}}\rightarrow\tilde S_{k},\qquad
    \widetilde{S_{k+1}\setminus D_{k,+}}\rightarrow\tilde S_{k+1},\qquad
  \tilde S_k\rightarrow \tilde{M}
\]
so that the corresponding inclusions of fundamental groups
$\pi_1(S_k\setminus D_{k,\bullet})<\pi_1(S_k)<\pi_1(M)$ make the following
diagram commutative
\begin{equation}\label{eq:lif}
\begin{CD}
\pi_1(S_k\setminus D_{k,-})@>(i_{e_k})_*>>\pi_1(S_{k+1}\setminus D_{k,+})\\
@VVV@VVV\\
\pi_1(M)@=\pi_1(M)
\end{CD}
\end{equation}

Notice that when the liftings $\tilde S_1\setminus D_{1,-}\rightarrow\tilde S_1$ and
$\tilde S_1\rightarrow M$ are chosen, all the other liftings are fixed by the commutativity of
(\ref{eq:lif}).

An inductive argument, based on the van Kampen theorem, shows that the induced map
\[
\pi_1(S_1)*\pi_1(S_2)*\ldots*\pi_1(S_n)\rightarrow\pi_1(M)
\]
is surjective with kernel generated by elements
$(i_{e_k})_*(\gamma)\gamma^{-1}$ for $\gamma\in\pi_1(S_k\setminus D_{k,-})$.

In particular given a geometric data $\mu=(\mu_l(v), \mu_r(v))$, we may fix the holonomy representations
of the left and right metrics so that they determine a representation
\[
 H=H(\mu):  \pi_1(M)\rightarrow PSL_2(\mathbb R)\times PSL_2(\mathbb R)~.
\]

More precisely we fix the holonomies of $\mu_l(v_1), \mu_r(v_1)$, say $h_l^1, h_r^1$,
in their conjugacy classes. Then, we can fix recursively
the holonomies of $\mu_l^k, \mu_r^k$ in their conjugacy classes
so that
\[
h_\bullet^k(\gamma)=h_{\bullet}^{k+1}((i_{e_k})_*(\gamma))
\]
for all $\gamma\in\pi_1(S_{k}\setminus D_{k,-})$.
Notice that once $h_\bullet^1$ is fixed, all the other representations
are uniquely determined, since the  holonomy  $S_v\setminus D$
is not elementary.
In particular,  though the representation $H$ depends on some choices
(including the isomorphism between $\pi_1(M)$ and the quotient of the
free product of $\pi_1(S_i)$), its conjugacy class is well defined.

By definition $hol=H\circ GD$.
So in order to conclude the proof we need to prove that
\begin{itemize}
\item $H$ is injective;
\item $H$ takes value in $\mathcal R(g, T, \theta)$ (that is, in the space of admissible representations
as in Definition \ref{df:adrep});
\item  $H$ is an open map.
\end{itemize}

The first point easily follows from Lemma \ref{lm:holonomy}, since the representation
$H(\mu)$ contains all the holonomies of the metrics $\mu(v_i)$.

In order to prove that $H(\mu)$ is admissible, we need to check that
its restriction to the fundamental group of the link of any collision
point is conjugated to a diagonal representation. 
As we will see, this is essentially a consequence  of
property (3) of the definition of double surgery (as in Definition \ref{df:double}).

In fact, fix a collision point $p$, and suppose that $S_k, S_{k+1}$ are the surfaces
separated by $p$.
Fix a lifting of  the natural inclusions $\tilde D_{k,-}\rightarrow\tilde S_{k}$ so that:
\begin{itemize}
\item the intersection of the closure of $\tilde D_{k,-}$ and
$\widetilde{S_{k}\setminus D_{k,-}}$ is not empty and it corresponds to a
lifting $\tilde \partial_-$ of $\partial D_{k,-}$.
\item
Analogously the intersection of the closure of $\tilde D_{k,+}$ and
$\widetilde{S_{k+1}\setminus D_{k,+}}$ corresponds to a
lifting $\tilde \partial_+$ of $\partial D_{k,-}$.
\item
the stabilizers of $\tilde \partial_-$ and $\tilde \partial_+$ in $\pi_1(M)$ are the same
$\mathbb Z$ subgroup generated by $\gamma_0$.
\end{itemize}

Notice that with these choices the fundamental group of the link $\Sigma$ of $p$
is generated by $\pi_1(\tilde D_{k,+})$ and $\pi_1(\tilde D_{k,-})$.
Analogously if $\Omega$ is the union of adjacent spacial slices corresponding to
$v_k$ and $v_{k+1}$, its fundamental group in $\pi_1(M)$ is generated by $\pi_1(S_k)$
and $\pi_1(S_{k+1})$

Changing the metrics $\mu_l(v_k), \mu_r(v_k)$ and $\mu_l(v_{k+1}), \mu_r(v_{k+1})$ 
by some isotopy, we can require that they coincide on $D_{k,-}$ and $D_{k,+}$ respectively.
Let $S_{c,-}, S_{c,+}$ be the surfaces obtained by collapsing respectively
$D_{k,-}$ on
$(S_{k},\mu_l(v_k), \mu_r(v_k))$ and $D_{k,+}$ on $(S_{k+1},\mu_l(v_{k+1}), \mu_r(v_{k+1}))$.

Finally choose liftings of the isometries $\phi_l, \phi_r:S_{c,-}\rightarrow S_{c,+}$, say
$\tilde\phi_l, \tilde\phi_r$, so that
\[
(\phi_\bullet)^{-1}\circ\gamma\circ\phi_\bullet=\gamma
\]
for all $\gamma\in\pi_1(S_{c,-})=\pi_1(S_k\setminus D_{k,-})=\pi_1(S_{k+1}\setminus D_{k,+})=\pi_1(S_{c,+})$.

We fix now
\begin{enumerate}
\item developing maps $dev^-_l, dev^-_r:\tilde S_{k}\rightarrow\mathbb H^2$
of $\mu_l(v_k), \mu_r(v_k)$ so that they coincide on $\tilde D_{k,-}$,
\item developing maps $d^-_l, d^-_r:\tilde S_{c,-}\rightarrow\mathbb H^2$
extending $dev^-_\bullet(v_k)$ on $\widetilde{S_k\setminus D_{k,-}}$,
\item developing maps $d^+_l, d^+_r:\tilde S_{c,+}\rightarrow\mathbb H^2$ defined
as $d^+_\bullet=d^-_\bullet\circ(\tilde\phi_\bullet)^{-1}$,
\item developing maps $dev^+_l, dev^+_r:\tilde S_{k+1}\rightarrow\mathbb H^2$
which extend $d^+_\bullet $ on $\widetilde{S_{k+1}\setminus D_{k,+}}$.
\end{enumerate}

Notice that the holonomies of $dev^\pm_l, dev^\pm_r$ glue to a representation
$H_\Omega: \pi_1(\Omega)\rightarrow PSL(2,\mathbb R)\times PSL(2, \mathbb R)$
which is conjugated to $H|_{\pi_1(\Omega)}$.

So it is sufficient to prove that $(H_\Omega)|_{\pi_1(\Sigma)}$ is a diagonal representations.
Since we are assuming that $dev^-_l$ and $dev^-_r$ coincide on $\tilde D_{k,-}$,
 it is sufficient to prove that $dev^+_l$ and $dev^+_r$ coincide on $\tilde D_{k,+}$.
Since $\mu_l(v_{k+1})$ coincides with $\mu_r(v_{k+1})$ on $D_{k,+}$, there exists
$\tilde R\in PSL(2,\mathbb R)$ such that
\begin{equation}\label{eq:kk}
dev^+_r=\tilde R\circ dev^+_l
\end{equation}
on $\tilde D_{k,+}$.

On the other hand, since $\partial D_{k,+}$ bounds a disk
in $S_{c,+}$ where the left and right metrics coincide, by condition (3)
of Definition \ref{df:double}, $\phi_l$ and $\phi_r$ coincide on $\partial D_{k,+}$ so
$\tilde\phi_l$ and $\tilde\phi_r$ coincide on $ \tilde\partial_+$.
So $d^+_l$ and $d^+_r$ coincide on $\tilde\partial_+$.
In particular $dev^+_r=dev^+_l$ on $\tilde\partial_+$, and this with (\ref{eq:kk}) implies that
$dev^+_r$ and $dev^+_l$ coincide on the whole $\tilde D_{k,+}$.
This finishes the proof that $H\in \mathcal R(g, T, \theta)$.

Finally we need to check that the map $H:\mathcal D(X)\rightarrow \mathcal R(g, T, \theta)$ is
open. Given $H'$ close to $H(\mu)$,  the representation
$H'|_{\pi_1(S_v)}$ is close to $H(\mu)|_{\pi_1(S_v)}$ so it is a pair of holonomies of hyperbolic
structures with cone angles $\mu'_l(v), \mu'_r(v)$. Since the trace of the $H'$-image of peripheral
elements of $\pi_1(S_v)$ is fixed, the cone angle at each point $x_i$ is just
$\theta(x_i)$.
Now in order to conclude we need to check that if $e=[v_-, v_+]$ is an edge of the graph
of the manifold $M$, $(S_{v_+},\mu'_l(v_+), \mu'_r(v_+))$ is obtained by a double surgery
on $(S_{v_-}, \mu'_l(v_-), \mu'_r(v_-))$ with surgery disks isotopic to $D_{v_-,-}$ and
$D_{v_+, +}$ and identification maps isotopic to $i_e$.

This fact can be easily proved by the same argument used in the proof of Proposition \ref{pr:collision}.
\fb{ Let $(S_0,\mu^0_l,\mu^0_r),$ and $(\hat S_0,\hat\mu^0_l, \hat\mu^0_r)$ 
be the surfaces obtained respectively from
 $(S_{v_+},\mu'_l(v_+), \mu'_r(v_+))$ and $(S_{v_-}, \mu'_l(v_-), \mu'_r(v_-))$
 by collapsing $D_{v_+,+}$ and $D_{v_-,-}$.
Notice that the holonomy of $(S_0,\mu^0_\bullet)$ coincides with the holonomy
of $(S_{v_+}\setminus D_{v_+,+},\mu_\bullet(v_+))$ and analogously the holonomy
 of $(\hat S_0,\hat\mu^0_\bullet)$ coincides with the holonomy
of $(S_{v_-}\setminus D_{v_-,-},\mu_\bullet(v_-))$.
Since $(S_{v_+}\setminus D_{v_+,+},\mu_\bullet(v_+))$ is isotopic to
 $(S_{v_-}\setminus D_{v_-,-},\mu_\bullet(v_-))$ in $M$ we deduce that
 the holonomies of $(S_0,\mu^0_\bullet)$ and $(\hat S_0, \hat\mu^0_\bullet)$
 are conjugate. Thus, by Lemma \ref{lm:holonomy}, the collapsed surfaces 
$S_0$ and $\hat S_0$ are isometric for both the left and right metrics.}

The fact that $H'$ restricted to the fundamental group of the link of the collision point between
$S_{v_-}$ and $S_{v_+}$ is diagonal implies that the condition (3) in the definition of
double surgery is satisfied.
\end{proof}

\appendix

\section{An example where a double surgery is needed}
\label{ap:A}

Let $S_1$ and $S_2$ be two hyperbolic surfaces with cone singularities of angles less than $2\pi$ 
and let $D_1$ and $D_2$ be two disks embedded in $S_1$ and $S_2$ respectively. We suppose that
there is a diffeomorphism preserving cone points $f:S_1\setminus D_1\rightarrow S_2\setminus D_2$
such that the holonomy of $S_1\setminus D_1$ is conjugated to the representation obtained by
composing the holonomy of $S_2\setminus D_2$ with the map 
$f_*:\pi_1(S_1\setminus D_1)\rightarrow\pi_1(S_2\setminus D_2)$ induced by $f$.
We will also assume in this appendix that the holonomy of $\partial D_2$ is elliptic of angle 
$\theta<2\pi$.

In Section \ref{ssc:surgeries} we have shown that collapsing $D_1$ in $S_1$ and $D_2$ in $S_2$
yields the same surface (up to isometry). This means that there are two surgeries
involved in the transformation from $S_1$ to $S_2$. First the disk $D_1$ is replaced by a disk $P_1$ containing
only a cone point, \jms{yielding} a surface $S$.  Then another disk $P_2$ of $S$ isotopic to $P_1$ is replaced by $D_2$.
We could expect at first glance that a single surgery would be sufficient; for example, that choices of disks of 
surgery can be made so that $P_1$ and $P_2$ coincide. 
But in this section we will prove that this cannot always be the case. 
We will find a criterion to establish whether a single surgery
is sufficient, and will then construct an example where this criterion fails.

Notice that the complement of any disk isotopic to $D_\bullet$ in $S_\bullet$ isometrically
embeds in $S$. We consider the minimal convex disk $\Delta$ 
isotopic to $D_\bullet$ constructed in  Proposition \ref{pr:isomdsk}.
(Notice that if $D_\bullet$ contains only $2$ singular points, the $\Delta$
degenerates to a segment, but the argument below can be adapted.)

We still have that $S_\bullet\setminus\Delta_\bullet$ embeds in $S$.
According to the following definition, the complement of its
image is a polygon with center $p$ which we denote by $P_\bullet$.

\begin{defi}
Let $S$ be a hyperbolic surface with cone singularities, and $p$ be a cone point on $s$.
A polygon with center $p$ in $S$, is an embedded disk $P$, such that
\begin{itemize}
\item $p$ is in its interior;
\item $P$ is the union of  hyperbolic triangles which have all a vertex at $p$.
\end{itemize}
\end{defi}

\begin{prop}\label{criterion:prop}
$S_2$ is obtained by a single surgery on $S_1$ if and only if $P_1$ and $P_2$ are
contained in a disk $\Pi$ embedded in $S$.
\end{prop}

The if part is easy: $S\setminus \Pi$ embeds in $S_\bullet$  and its complement
is a disk $\Delta'_\bullet$ isotopic to $D_\bullet$. Thus we can cut from $S_1$ the disk
$\Delta'_1$ and glue instead the disk $\Delta'_2$. The surface that we obtain is obviously isometric to $S_2$.

The only if part easily follows from the following lemma.

\begin{lemma}\label{criterion:lm}
If $S_2$ is obtained by a single surgery on $S_1$, then the surgery can be done replacing
a convex  disk $D'_1$ whose boundary is piecewise geodesic with some vertices
at cone points and some vertices in the smooth part (which can degenerate to segment)
by a disk $D'_2$ with similar properties.
 \end{lemma}

\begin{proof}
We denote by $\mathcal D$ the set of disks embedded in $S_1$ isotopic to $D_1$,
whose complement can be embedded in $S_2$ by an isometric map isotopic to $f$.

Take a sequence of disks $D_n\in\mathcal D$ which minimizes the length of the boundary.
\fb{
Take a parameterization $c_n:[0,1]\rightarrow\partial D_n$. Up to taking a subsequence,
$c_n$ converges to a curve $c_\infty$ 
 and $D_n$ converges to  some subset $D'_1$.
By the minimality it turns out that $c_\infty $ is} \jms{a} \fb{piecewise geodesic curve with vertices
at cone points. Moreover if $c_\infty(t_0)$ is a cone point, then the segments
$[c_\infty(t_0)-\epsilon, c_\infty(t_0)]$ and $[c_\infty(t_0), c_\infty (t_0+\epsilon)]$}
\jms{form} \fb{an angle bigger than $\pi$ in $S\setminus D'_1$.

It follows that there are two cases: either $c_\infty$ spans a segment with vertices
at two cone points, or it is an embedded curve. In the first case $D'_1$ coincides
with the support of $c_\infty$, whereas in the second case it is a convex disk bounded
by $c_\infty$.
}


In both cases $S_1\setminus D'_1$ embeds in $S_2$. The complement $D'_2$ of the embedding
of $S_1\setminus D'_1$ in $S_2$ is still a convex disk with piecewise geodesic or a segment.
Clearly $S_2$ is obtained from $S_1$ by cutting $D'_1$ and replacing $D'_2$.
\end{proof}

We can now prove that the condition of Proposition \ref{criterion:prop} is necessary.
Let $D'_1$ and $D'_2$ be as in Lemma \ref{criterion:lm}.
The minimal disk $\Delta_\bullet$ is contained in $D'_\bullet$. 

It follows that $S_1\setminus D'_1$ and $S_2\setminus D'_2$ are both isometric
to a region $S'$ of $S$, whose complement --- say $\Pi$ --- is a polygon with vertex at $p$.
Moreover $S'$ is contained in $S\setminus P_1$ and $S\setminus P_2$, so $\Pi$ contains
both $P_1$ and $P_2$, and this concludes the proof of Proposition \ref{criterion:prop}.

\medskip

In the remaining part of this section we will construct an example of two surfaces $S_1$ and $S_2$
satisfying the condition given at the beginning of the appendix but such that the disks $P_1$ and $P_2$ in $S$ are not contained in a disk. 
This will show that $S_2$  cannot be obtained by a single surgery on $S_1$.

First we will construct a surface $S$ containing only a single cone point, then we will
find two convex polygons $P_1$ and $P_2$ around the cone point whose union is contained in
no embedded disk. Finally, making a surgery on $P_\bullet$, we will construct two surfaces $S_\bullet$
such that the complement of the corresponding minimal disk is isometric to the complement of $S\setminus P_\bullet$.

\subsubsection*{Construction of $S$ and $P_\bullet$}

We consider in $\mathbb H^2$ a regular convex  octagon $Q$ such that the sum of its interior angles is $\theta \in (0,2\pi)$. 
Gluing opposite sides of $Q$ we obtain a hyperbolic surface $S$ with a cone point
corresponding to the vertices of $Q$ whose cone angle is $\theta$. Denote by $\pi:Q\rightarrow S$
the projection map.

Let $l$ be the length of any edge of $Q$. Choosing $l'<l/2$, we can consider for every vertex  $p_i$
of $Q$ the triangle  $T_i$ with two edges of length $l'$ contained in the edges of $Q$ at $p$. 
The union of those triangles projects to a convex polygon  $P_0$ with center $\bar p$ in $S$.
Notice that the angle of this polygon is equal to $2\phi$, where $\phi$ is the base angle
of the triangles. We now add to $P_0$ a triangle $Z_i$ with a vertex at the center of $Q$ and 
opposite edge equal to an edge  $e_i$ of a triangle $T_i$. If $l'$ is close to $l/2$
then the sum of the angle of $Z_i$ at a vertex of $e_i$ and $2\phi$ is less than $\pi$,
so $P_i=P_0\cup Z_i$ is a convex polygon contained in $S$. 
Since $P_1\cup P_2$ contains a non-trivial loop, it is not contained in any disk of $S$.

\subsubsection*{Construction of $S_1$ and $S_2$}

Notice that $P_1$ is a polygon with center $\bar p$ in $S$. 
In fact $T_1\cap Z_1$ is the union of two triangles with a vertex at $\bar p$. 
In particular $P_1$ can be decomposed as a union of
triangles with vertex at $\bar p$. Each of these triangles has a boundary edge and two interior edges.
There is one interior edge joining $\bar p$ to the center of $Q$ whose length is $u$, whereas
the length of all the other interior edges is $l'$.

 Choose $\epsilon$ small. We can deform
 every triangle of the decomposition of $P_1$ without changing its boundary edge and shortening
 the other two edges of $\epsilon$. Let us call $P_1'$ the polygon obtained in this way.
 Replacing $P_1$ by $P_1'$ we get a surface $S_1$ with a cone point at each vertex of $P_1'$ and
 a central vertex.
 
 Analogously we can obtain a surface $S_2$ by making a surgery on $P_2$.
 It is not difficult to construct a diffeomorphism $f: S_1\setminus P_1'\rightarrow S_2\setminus P_2'$
 such that the following diagram is homotopically commutative.
 
To conclude it is sufficient to notice that $P_\bullet'$ is a minimal disk whose complement is
isometric to $S\setminus P_\bullet$. Since there is no disk containing both $P_1$ and $P_2$,
Proposition \ref{criterion:prop} implies that $S_2$ cannot be obtained by a single surgery 
on $S_1$.

\bibliography{bibliocollision}
\bibliographystyle{alpha}

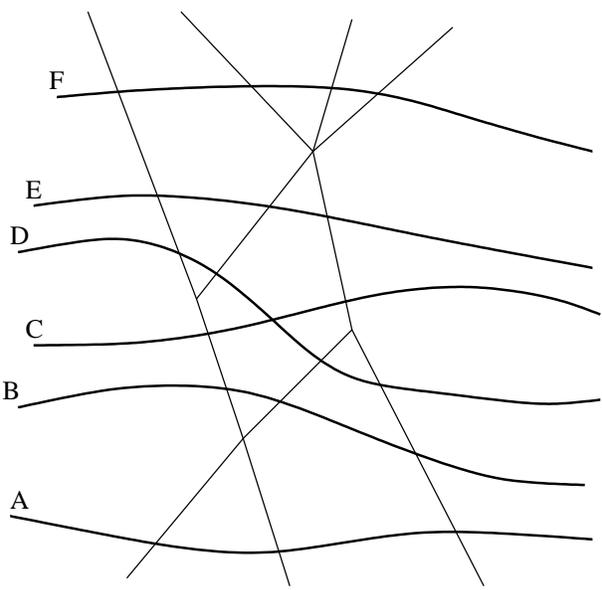
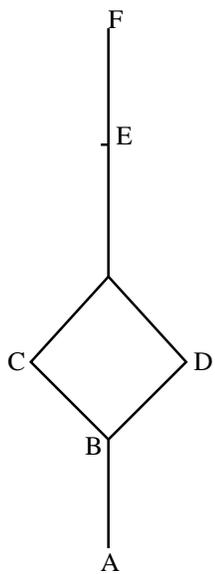
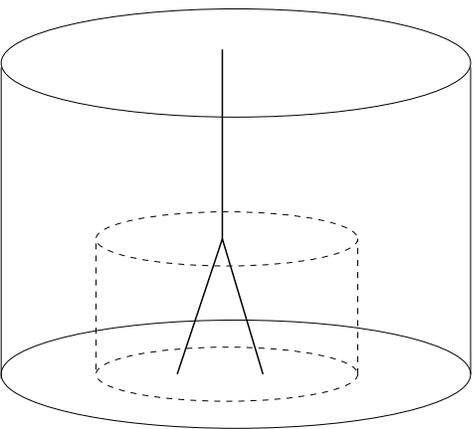
\end{document}